\documentclass[a4paper,leqno,12pt]{amsart}
\usepackage{amsmath,amsthm}
\usepackage{amssymb}
\usepackage{array}    
\usepackage[all]{xypic}
\usepackage[all,cmtip]{xy}
\usepackage{tikz}
\usepackage{tikz-cd}
\usepackage{mathtools}
\usetikzlibrary{matrix,shapes,arrows,positioning,chains}
\usepackage{url}
\usepackage{hyperref}
\usepackage{enumitem}
\usepackage{rotating}
\usepackage{lscape}
\usepackage{multirow}
\newlist{steps}{enumerate}{1}
\setlist[steps, 1]{label = Step \arabic*:}

\SelectTips{cm}{12}
\UseTips{}
\usepackage{euscript}
\begin{document}

\theoremstyle{plain}
\swapnumbers
 \newtheorem{theorem}{Theorem}[subsection]
\newtheorem{lemma}[theorem]{Lemma}
\newtheorem{corollary}[theorem]{Corollary}
\newtheorem{proposition}[theorem]{Proposition}

\newtheorem*{thma}{Theorem A}
\newtheorem*{thmb}{Theorem B}
\newtheorem*{thmc}{Theorem C}
\newtheorem*{thmd}{Theorem D}

\theoremstyle{definition}
\newtheorem{definition}[theorem]{Definition}
\newtheorem{example}[theorem]{Example}
\newtheorem{prop}[theorem]{Proposition}
\newtheorem{problem}[theorem]{Problem Statement}
\newcommand{\R}{\mathbb{R}}
\newcommand{\Z}{\mathbb{Z}}
\newcommand{\mF}{\mathbb{F}}
\theoremstyle{remark}
\newtheorem{ob}[theorem]{Observation}
\newtheorem{remark}[theorem]{Remark}
\newtheorem{ack}[theorem]{Acknowledgement}

%
%
\title[K ring of the Flip Stiefel Manifold]{The complex $K$ ring of the flip Stiefel manifolds}
\author[S.~Basu]{Samik Basu}
\address{Stat-Math Unit,
Indian Statistical Institute,
B. T. Road, Kolkata-700108, India.}
\email{samik.basu2@gmail.com; samikbasu@isical.ac.in}
\author[S.~Gondhali]{Shilpa Gondhali}
\address{Department of Mathematics\\ Birla Institute of Technology and Science (BITS)-Pilani, K K Birla
Goa Campus\\ 403726 Goa, India.}
\email{shilpag@goa.bits-pilani.ac.in, shilpa.s.gondhali@gmail.com}
\author[S.~Fathima]{Fathima Safikaa}
\address{Department of Mathematics\\ Birla Institute of Technology and Science (BITS)-Pilani, K K Birla
Goa Campus\\ 403726 Goa, India.}
\email{p20180409@goa.bits-pilani.ac.in, fsafikaa@gmail.com }
\date{\today}

\begin{abstract}
The flip Stiefel manifolds \(FV_{m,2s}\) are defined as the quotient of the real Stiefel manifolds \(V_{m,2s}\) induced by the simultaneous pairwise flipping of the co-ordinates by the cyclic group of order 2. We  calculate the complex \(K\)-ring of the flip Stiefel manifolds,  $K^\ast(FV_{m,2s})$, for $s$ even. Standard techniques involve the representation theory of $Spin(m),$ and the Hodgkin spectral sequence. However, the non-trivial element inducing the action doesn't readily yield the desired homomorphisms. Hence, by performing additional analysis, we settle the question for the case of \(s \equiv 0 \pmod 2.\) 
\end{abstract}
\maketitle

\newcommand{\mR}{{\mathbb R}}
\newcommand{\Res}{\text{Res}}
\newcommand{\mC}{{\mathbb C}}
\newcommand{\diag}{\text{diag}}
\newcommand{\nf}[2]{N_F({#1},{#2})}
\newcommand{\np}[2]{N_P({#1},{#2})}
\newcommand{\mZ}{{\mathbb Z}}
\newcommand{\mS}{{\mathbb S}}
\newcommand{\mRP}{{\mathbb RP}}
\newcommand\restr[2]{\ensuremath{\left.#1\right|_{#2}}}
\section{Introduction}\label{sintro}
The study of topological properties of the flip Stiefel manifold was initiated in \cite{bgs}. The flip Stiefel manifold is formally defined as follows:

\begin{definition}
Let $1 < 2s \leq m$. Consider an action of $C_2= \{e, a\},$ the additive group of order two, on $V_{m, 2s}$ given by
\[
a \cdot (v_1, v_2, \ldots, v_{2s})= (v_2, v_1, v_4, v_3, \ldots, v_{2s}, v_{2s-1})
\]
where $(v_1, v_2, \ldots, v_{2s}) \in V_{m, 2s}.$ A {\it{flip Stiefel manifold}}, denoted by the symbol $FV_{m,2s}$, is the quotient space of $V_{m,2s}$ under the given action of $C_2.$
\end{definition}
Equivalently, we can express the action of $a$ on $V=(v_1, v_2, \ldots, v_{2s})$ as multiplication from right by the $2s \times 2s$  permutation matrix  $F=\diag(f, f,\ldots, f)$ for $f= \begin{bmatrix}
0 & 1\\
1 & 0
\end{bmatrix}.$ To be precise,
\[
a \cdot V=
\begin{matrix}
\begin{matrix} v_1 & v_2 &\cdots & v_{2s} \end{matrix} & \\
\begin{bmatrix}
v_{11} & v_{12} & \cdots & v_{12s} \\
v_{21} & v_{22} & \cdots & v_{22s}\\
\vdots & \ddots & \ddots & \vdots \\
v_{m1} & v_{m2} & \cdots & v_{m2s}
\end{bmatrix} &
\begin{matrix}  \end{matrix}&
\begin{matrix} F \end{matrix}
\end{matrix}\]

The investigation into the topology of \(FV_{m,2s}\) initiates with the analysis of its tangent bundle, in \cite{bgs}. By computing the Stiefel-Whitney classes and mod-2 cohomology of \(FV_{m,2s}\), the authors derived several results related to parallelizability, span and stable span of \(FV_{m,2s}\). In addition to this analysis, the authors demonstrated that \(FV_{m,2s}\) is a homogeneous manifold and discussed index theory based applications to equivariant maps between Stiefel manifolds. In this paper, we extend the analysis of the topology of \(FV_{m,2s}\), a step further, by determining \(K^{\ast}(FV_{m,2s})\) for \(s \equiv 0 \pmod 2.\) \\

\noindent Among the various historical approaches employed in computing the complex \(K\)-ring of a homogeneous  manifold \(G/H\), where \( G\) denotes a compact, connected Lie group with  torsion free fundamental group and \(H\) denotes a closed subgroup of \(G\), two particularly efficient tools stand out: the Atiyah-Hirzebruch spectral sequence \cite[Theorem 3.6]{ah} and the Hodgkin spectral sequence \cite{h}. While the Atiyah-Hirzebruch spectral sequence requires the closed subgroup \(H\) to be connected and maximally ranked in \(G\), the method by Hodgkin allows us to work with a more general \(H.\) Thus, from \cite[Lemma 4]{bgs}, we get the insight to proceed with Hodgkin's approach. We state one of the  versions of the Hodgkin spectral sequence here and for detailed discussions, refer to \cite{aguz}, \cite[Section 3]{g} and \cite{ro}.
 
\begin{theorem}(Hodgkin) Let $G$ be a compact, connected Lie group with $\pi_1(G)$ torsion-free and H be a closed subgroup of G. There exists a strongly convergent spectral sequence such that
\begin{itemize}
\item As an algebra, $E^p_2(G/H) = Tor^p_{RG}(RH; \mZ).$
\item The differential $d_r: E^{p-r}_r \rightarrow E^p_r$ is zero for $r$ even.
\item $E_{\infty}^{\ast}$ is the graded algebra associated to a multiplicative filtration of $K^{\ast}(G/H),$ i.e., there is a filtration 
\[ 0 = F_{-2} \subset F_0 \subset F_2 \subset \cdots \subset F_{2n} = K^{0}(G/H)\]
\[0 = F_{-1} \subset F_1 \subset F_3 \subset \cdots \subset F_{2n+1} = K^{1}(G/H) \]
with $F_i \cdot F_j \subset F_{i+j}$ under the product in $K^{\ast}(G/H)$ and $E^p_{\infty} = F_p/F_{p-2}$. Furthermore, the product in $E^{\ast}_{\infty}$ is that induced naturally by that in $K^{\ast}(G/H).$
\end{itemize}
\end{theorem}

To make use of this theorem, we choose to view $FV_{m,2s}$ as a quotient of the compact, connected Lie group \(Spin(m)\), whose fundamental group is torsion-free, by an appropriate closed subgroup, say $H_{m,2s}$. Then  we determine the $RSpin(m)$-module structure of $RH_{m,2s}$ by evaluating the restriction homomorphism $\Res \colon RSpin(m) \to RH_{m,2s}$, where $RSpin(m)$ and $RH_{m,2s}$ denote the complex representation rings of $Spin(m)$ and $H_{m,2s}$ respectively. Finally, we compute the most essential ingredient of the Hodgkin spectral sequence, $Tor^\ast_{RSpin(m)}(RH_{m,2s}; \mZ)$, that describes \(K^{\ast}(FV_{m,2s}).\) \\
 
This approach represents a standard methodology, widely employed in the calculations of the complex $K$-ring of the other quotients of Stiefel manifolds like the real projective Stiefel manifold and complex projective Stiefel manifold(\cite{aguz}, \cite{bh}, \cite{g}). It is crucial to note that these manifolds are obtained as quotient spaces under the action of right multiplication by the matrix $D= diag(z, z, \ldots, z)$ where $z \in \mathbb{Z}/2$ for the real projective Stiefel manifold and $z  \in S^1$ for the complex projective stiefel manifold. However, a significant challenge posed by the flip Stiefel manifold is that unlike the matrix $D$ that belongs to the center of $O(m)$ or $U(m)$ in their respective cases, the matrix $F$ is not a central element. These matrices, essentially defining the action, play a pivotal role in the  \(K\)-ring computation because it is through these matrices that we obtain the description of the quotienting subgroup \(H\),  subsequently the ring \(RH\) and the \(RG\)-module structure of \(RH\). The lack of centrality in \(F\) necessitates  remedying modifications to be incorporated to seamlessly carry out the computations. \\

\noindent Before outlining the section-wise plan, we fix a few notations which are necessary because our calculations involve representation ring of the {\it Spin} group which has distinct description in odd and even cases.
\[
\begin{tabular}{ccc}
  & ~ $\bf  m= 2n$ or $\bf 2n+ 1,$ & ~$\bf m-2s = 2c$ or $\bf 2c+ 1.$ 
\end{tabular}
\]

In Section  \ref{sfvandspin}, we justify our choice of \(G=Spin(
m)\) and determine the exact description of \(H_{m,2s}\) such that \(FV_{m,2s}\) can be viewed as the homogeneous space \(G/H_{m,2s}\),   meeting the conditions specified by the Hodgkin spectral sequence theorem. In Section \ref{srep_rings}, we recall the representation ring of \(Spin(m)\), denoted \(RSpin(m)\) \cite{hu}), and we describe the role played by the maximal tori of \(Spin(m)\) in determining the representation ring of the closed subgroup \(H_{m,2s}\). In Section \ref{smoddsequiv0}, we demonstrate the computation of \(K^{\ast}(FV_{m,2s})\) for \(m \equiv 1 \pmod 2\) and \(s \equiv 0 \pmod 4\) in a detailed manner. This computation gives us Theorem A, one of the primary results of this paper. Here, and in the rest of the results on \(K^{\ast}(FV_{m,2s})\), we follow the convention that relations involving $u_4$ do not appear in $I$ when $\alpha = n.$
\begin{thma}
For \(m \equiv 1 \pmod 2 \) and \(s \equiv 0 \pmod 4\), let $2^{\alpha}= GCD\{2^n, b_0\}$ with $b_0$= $GCD\displaystyle{\bigg\{\binom{(s/2) + i -1}{i}2^{2i-1}\bigg|i= c+1, c+2, \ldots, n-1\bigg\}}$ and 
\[
K^\ast(FV_{m,2s}) \cong \begin{cases}
\Lambda^{\ast}[t_1, t_2, \ldots t_{s-2}, u_1, u_2, u_3,v] \otimes \mZ[y,\delta_c]/I \hspace{0.9cm}\text{when} ~\alpha= n,\\
\Lambda^{\ast}[t_1, t_2, \ldots t_{s-2}, u_1, u_2, u_3,u_4,v] \otimes \mZ[y,\delta_c]/I  ~~~\text{when} ~\alpha< n,
\end{cases}
\]
where \(I\) is generated by $y^2+2y, ~\delta_c^2+ 2^{c+1}\delta_c- 2^{2c-1}y \bigg[1- \binom{(s/2)+c}{c} \bigg], ~2^{\alpha}y, ~2^{s-1}(2+y)\delta_c+ 2^{n-1}y, ~2^{\alpha}u_3 - (2^{s-1} \delta_c + 2^{n-1})u_1, ~(y+2)u_3- 2^{n-\alpha}u_1, ~yu_1, ~yu_2, ~\delta_cu_2, ~(y+2)u_4- 2^{n-1-\alpha}\delta_cu_1- 2^{-\alpha}b_0u_2, ~(y+2)(2^c\delta_c+1)u_3 +(y+2)\delta_cu_4-2^{n-\alpha}u_1, ~2^{c-\alpha+1}b_0u_2 + (y+2)(2^{c+1} +\delta_c)u_4, ~2^{c-\alpha -1}b_0u_2 + (2^{c-1}\delta_c + 2^{2c-2}y\bigg[1- \binom{(s/2)+c}{c} \bigg] +2^{c-1}\delta_cy)u_3 + (2^c+\delta_c+2^cy+\delta_cy)u_4, ~u_2u_4, ~u_1u_3, ~\delta_cu_1u_4, ~2^{n-\alpha}b_0u_1u_2+2^nu_1u_4+b_0u_2u_3, ~2^{n-2\alpha}b_0u_1u_2+2^{-\alpha}b_0u_2u_3 +2u_3u_4, ~-2^{n-\alpha+1}u_1u_4 + 2^{-\alpha}b_0u_2u_3+2u_3u_4, ~2^{n-2\alpha}b_0u_1u_2 + 2^{n-\alpha}u_1u_4+2u_3u_4, u_1u_2- 2v.$
\end{thma}
In Sections \ref{smevensequiv0} and  \ref{smoddsequiv2}, we reiterate the procedure followed in Section \ref{smoddsequiv0} by making necessary modifications to obtain \(K^{\ast}(FV_{m,2s})\) in the remaining cases under consideration. 

\begin{ack}
Second author thanks Prasanna Kumar for carefully proof-reading the manuscript. Third author thanks University Grant Commission - Ministry of Human Resource Development, New Delhi for the financial support provided through the CSIR- UGC fellowship. Third author thanks the Birla Institute of Technology and Science Pilani, K K Birla Campus for the facilities provided to carry out research work. 
\end{ack}

\section{$FV_{m,2s}$ as the quotient of $Spin(m)$}\label{sfvandspin}
 In \cite{bgs}, it is established that \(FV_{m,2s}\) is a homogeneous manifold through the transitive action of the orthogonal group \(O(m)\). This is proven by demonstrating its equivalence to the quotient space \(\displaystyle \frac{O(m)}{C_2 \times O(m-2s)} \), where \(\displaystyle C_2= \{ I_{2s\times 2s},F\}\) and \(F\) is the matrix introduced in Section \ref{sintro}. It is to be noted that \(O(m)\) is not connected and thus it does not fulfill the requirement of the Hodgkin spectral sequence theorem. However,   the connected component   $SO(m)$ of \(O(m)\) also acts transitively on $FV_{m,2s}$ and similar argument proves $FV_{m,2s}  \cong \displaystyle \frac{SO(m)}{S(C_2 \times O(m-2s))}$, where $S(C_2 \times O(m-2s)) = SO(m) \cap (C_2 \times O(m-2s)),$ but \(\pi_1(SO(m))\) is not torsion-free. To address this,  we consider its double cover $Spin(m)$ of \(SO(m),\) which is connected and possesses a  torsion free fundamental group.\\

 Utilizing the 2-fold covering map $\varrho:Spin(m) \rightarrow SO(m)$, we establish $FV_{m,2s}  \cong \displaystyle \frac{Spin(m)}{H_{m,2s}}$, where $H_{m,2s} = \varrho^{-1}(S(C_2 \times O(m-2s)))$. To further elucidate the subgroup \(H_{m,2s}\), we recall the definition of $Spin(m)$ obtained by means of the real Clifford algebras \cite[Chapter 1, Section 6]{bd}. \\
 
 Consider $V=\mR^m$  equipped with the quadratic form $Q\colon \mR^m \to \mR$ defined by $Q(x) = -| x| ^2$. Let $e_1, \ldots, e_m$ denote the standard basis elements of \(V\). The Clifford algebra  $Cl(V,Q)$, denoted by \(\mathcal{C}_m\), with the structure map $i \colon \mR^m \to \mathcal{C}_m$ satisfying $(i(x))^2 = Q(x)\cdot1$ has the relations $e_i^2 = -1 $ and $e_i\cdot e_j = -e_j\cdot e_i$ for $i \neq j$. There is a canonical automorphism $\alpha: \mathcal{C}_m \to \mathcal{C}_m$  such that $\alpha^2 = id, ~ \alpha(x) = -x$ for $x \in i(V)$. Let $\mathcal{C}_m^{\ast}$ denote the group of units of $\mathcal{C}_m$. The subgroup $\Gamma_m = \{x \in \mathcal{C}_m^{\ast} \mid \alpha(x)vx^{-1} \in V,~ \text{for all} \, v \in V\}$ is endowed with the norm determined by \(N(x) = ~ |x| ^2\), for \(x \in \mR^m\). \(Spin(m)\) and \(Pin(m)\) are defined in terms of the homomorphism $p: \Gamma_m \to Aut(\mR^m)$ given by $p(x)v = \alpha(x)vx^{-1}$ as follows.
 \begin{itemize}
     \item $Pin(m)$ is the kernel of the norm \(N: \Gamma_m \to \mR^{\ast}\) and the map \(\restr{p}{Pin(m)}\) has image \(O(m).\)
\item $Spin(m)$ is the inverse image of $SO(m)$ under $p$ and \(\restr{p}{Spin(m)} = \varrho.\)
 \end{itemize}
 Having defined \(Spin(m)\) using the Clifford algebra \(\mathcal{C}_m\), we now move on to the description of \(H_{m,2s}.\)
\begin{lemma}\label{l2.0.1}
$H_{m,2s} = Spin(m-2s) \sqcup \omega Spin(m-2s)$ with 
$\omega = \omega_1\cdots \omega_{s}$, where $\omega_i = \frac{1}{\sqrt{2}}(e_{2i-1} - e_{2i})$, for all $i.$ 
\end{lemma}

The key ingredient of the proof of Lemma \ref{l2.0.1} is the following statement, whose proof is discussed in detail in \cite[Lemma 6.14]{bd}. 

 \begin{lemma}\label{l2.0.2}
 $\mR^m \setminus \left \{0\right\} \subset \Gamma_m,$ and if $u \in \mR^m\setminus \left \{0 \right \}$, then $p(u)$ is the reflection in the hyperplane orthogonal to $u$. Also $p\Gamma_m \subset O(m).$     
 \end{lemma}

 We proceed to the proof of Lemma \ref{l2.0.1}.

\begin{proof}
Recall that $H_{m,2s} = \varrho^{-1}(S(C_2 \times O(m-2s))$, where $S(C_2 \times O(m-2s))$ consists of all matrices of \(SO(m)\) of the form \( diag(M, M')\), where \(M \in C_2= \{ I_{2s\times 2s},F\}\) and \(M' \in O(m-2s)\).  From this representation, we observe that, \(Spin(m-2s) \subset H_{m,2s}\), where  \(Spin(m-2s)\) is generated by the last \((m-2s)\) vectors of the canonical basis of \(\mR^m\). Moreover, the identification of $FV_{m,2s}$ with $\displaystyle \frac{Spin(m)}{H_{m,2s}}$ via the double cover implies that \(Spin(m-2s)\) is an index 2 subgroup of \(H_{m,2s}.\)\\

\noindent To further obtain the missing component that makes up \(H_{m,2s}\), we need to determine \(p^{-1}(F).\) Since \(p\) is a homomorphism, it should be noted that it suffices to determine \(p^{-1}(f) = \begin{bmatrix}0 & 1 \\ 1 & 0 \end{bmatrix}.\) The matrix \(\begin{bmatrix}0 & 1 \\ 1 & 0 \end{bmatrix}\) is the reflection along the hyperplane $y=x $ in $\mR^2$. As per Lemma \ref{l2.0.2}, every non-zero vector $u$ in the line orthogonal to $y=x$ is mapped to \(f\) under the homomorphism \(p\). In particular, $\restr{p^{-1}}{Pin(2)}(f)$ is a point on the intersection of the unit circle and the line orthogonal to $y=x $, because \(Pin(2)\) consists only of unit normed vectors. That is, $\restr{p^{-1}}{Pin(2)}(f) \in \left \{\left (\frac{1}{\sqrt{2}}, -\frac{1}{\sqrt{2}}\right ), \left (-\frac{1}{\sqrt{2}}, \frac{1}{\sqrt{2}}\right)\right \}.$    Consequently, for \(s \equiv 0 \pmod 2\), the lemma is affirmed by the structure of \(Spin(m)\) as a subspace of \(Pin(m) \cap Cl_m^{even} \), where \(Cl_m^{even} = \mathcal{C}_m^0 \oplus \mathcal{C}_m^2 \oplus \mathcal{C}_m^4 \oplus \cdots\) is considered to be a vector space with natural grading.
\end{proof}

Having obtained \(H_{m,2s} \subset Spin(m)\), our next task is to explore the representation rings of the groups involved and the restriction homomorphism. Before we proceed with this analysis, we conclude this section with the following remarks about the element \(\omega \in H_{m,2s}\).

\begin{remark}\label{romegasquare}
\begin{enumerate}
\item[(1)] $\omega^2 = \pm 1.$ To be precise,
\[
\omega^2 = (\omega_1 \cdots \omega_{s})^2 = \begin{cases}
~1& \text{if} ~ s \equiv~ 0,3 \pmod 4\\
-1 & \text{if}~  s \equiv~ 1,2 \pmod 4 
\end{cases}
\]
\item[(2)] Although F is not a central element, \(\restr{p^{-1}}{Spin(m)}(F)\) is in the center of \(H_{m,2s}\) for the case in consideration, \(s \equiv 0 \pmod 2.\)
\end{enumerate}
\end{remark}

\section{Maximal Tori and Representation Rings}\label{srep_rings}

\noindent While the complex representations of \(Spin(m)\) are well-established in literature, determining the complex representations of \(H_{m,2s}\) demands careful attention. In general, for a compact, connected Lie group $G,$ its representations are obtained from those of its maximal torus (See  \cite{hu}).  In line with this theory, we study the maximal tori of \(Spin(m)\) with the aim of characterizing \(RH_{m,2s}\).

The standard maximal torus of $Spin(m)$ is given by 
\[
\tilde{T}= \{(\cos\theta_1 +e_1e_2 \sin\theta_1)\cdots(\cos\theta_n +e_{2n-1}e_{2n} \sin\theta_n)\mid 0 \leq \theta_i \leq 2\pi\}.
\]

Using this maximal torus, the complex representation ring of \(Spin(m)\) is derived as in \cite[Chapter 14, Theorem 10.3]{hu}. To ease the calculation of $Tor^{\ast}_{RSpin(m)}(RH_{m,2s}, \mathbb{Z}),$ we make use of the augmented (zero ranked) complex representations of $Spin(m)$. These representations have been formerly employed in the K-ring computation of the real projective Stiefel manifold as well. We refer to \cite[Page 32, Section 4]{aguz} and \cite[Page 3200, 6.2]{bh}. In Theorem \ref{trspin}, we describe 
 \(RSpin(m)\) in terms of these zero ranked complex representations. For completeness and clarity, we now explain the various symbols utilized in the theorem.\\
 
Let \(T^n\) denote the standard \(n\)-torus with \(RT^n = \mZ[z_1, z_1^{-1}, \ldots, z_n, z_n^{-1}] \), where \(z_i\) denotes the character of the projection of \(T^n\) onto the $i$-th \(\mS^1\) factor in \(\mS^1 \times \cdots \times~ \mS^1.\) 
\subsection{$RSpin(m)$ for $m \equiv 0 \pmod{2}$:}
\begin{enumerate}
\item Let $\Pi_k$ denote the $k$-th symmetrical function in the variables $\{z_i^2+ z_i^{-2}- 2\}$ for $i= 1, 2, \ldots, n.$ This is called {\it $k$-th Pontryagin class} \cite[Page~ 33]{aguz}. As mentioned earlier, we prefer \(\Pi_k\)s for our calculations over the standard exterior powers as \(\Pi_k\)s are in the kernel of the augmentation map \(RSpin(m) \to \mZ\).
\item Let
\[
\Delta_{n}^+= \sum\limits_{\substack{\epsilon_i= \pm1\\ \prod \epsilon_1= 1}} z_1^{\epsilon_1}\cdots z_n^{\epsilon_n}, \quad \Delta_{n}^-= \sum\limits_{\substack{\epsilon_i= \pm 1\\ \prod \epsilon_1= -1}} z_1^{\epsilon_1}\cdots z_n^{\epsilon_n} 
\]
\end{enumerate}
Then, we have
\[
R(Spin(2n))= \mathbb{Z}[\Pi_1, \ldots, \Pi_{n-2}, \Delta_{n}^+, \Delta_{n}^-].
\]

\subsection{$RSpin(m)$ for $m \equiv 1 \pmod{2}$:}
\begin{enumerate}
    \item Let $\Pi_i$ be as mentioned earlier.
    \item  Let
\[
\Delta_{n}= \prod_i(z_i+ z_i^{-1}).
\]
\end{enumerate}

Consequently, 
\[
R(Spin(2n+1))= \mathbb{Z}[\Pi_1, \ldots, \Pi_{n-1}, \Delta_{n}].
\]
Moreover, in \(RSpin(m)[t]\), we have
\begin{align*}
\Pi[t] &= \sum\limits_{i \geq 0} t^i \Pi_i= \prod\limits_{i=1}^n (1+ t(z_i+ z_i^{-1})^2),\\
\Delta_n[t] &= \prod\limits_{i=1}^n (z_i+ tz_i^{-1}).
\end{align*}
See \cite[Page 3191]{bh} for details. These generating functions enable us to establish the relationship between the generators of \(RSpin(m)\) efficiently. We encapsulate the preceding discussions on \(RSpin(m)\) within the following theorem.

\begin{theorem}(\cite[Page 3191]{bh})\label{trspin}
With the above notation,\\
(1) $RSpin(2n + 1)$ is the polynomial ring $\mathbb{Z}[\Pi_1,\ldots, \Pi_{n-1}, \delta_n]$, where
\[
\delta_n = \Delta_n-2^n,  \quad \Delta_n^2= \Pi_n+ 4\Pi_{n-1}+ \cdots+ 2^{2n}. 
\]
(2) $RSpin(2n)$ is the polynomial ring $\mathbb{Z}[\Pi_1,\ldots, \Pi_{n-2}, \delta_n^+, \delta_n^-]= [\Pi_1,\ldots, \Pi_{n-2}, \chi, \delta_n^+],$ where
\[
\delta_n^{\pm} = \Delta_n^{\pm} -2^{n-1},\quad \chi = \Delta_n^+ - \Delta_n^-, \quad \Delta_n^2= \Pi_n+ 4\Pi_{n-1}+ \cdots+ 2^{2n}, \quad \chi^2 = \Pi_n.
\]
\end{theorem}

\subsection{Representation ring of \(H_{m,2s}\)}\label{s3.3}
In Lemma \ref{l2.0.1}, we deduced that
\[
H_{m,2s} = Spin(m-2s) \sqcup \omega Spin(m-2s),~\text{where,}~\omega^2= \pm 1. 
\] Note that \(\omega\) is in the center of \(H_{m,2s}\). Additionally, for \(\omega^2=1\), for instance when \(m \equiv 1 \pmod 2 \) and \(s \equiv 0 \pmod 4\), we get a split short exact sequence, 
\[
\xymatrix
{
0 \ar[r] & Spin(m- 2s) \ar[r] & H_{m, 2s} \ar[r] & C_2 \ar[r] \ar@/_1.5pc/_{\omega}[l] & 0
}
\]
which gives us
\[
H_{m,2s} \cong C_2 \times Spin(m- 2s).
\]
An immediate consequence of this observation is
\[
RH_{m,2s} \cong RC_2 \otimes RSpin(m- 2s).
\]
But for \(\omega^2 = -1\) , for instance \( s\equiv 2 \pmod 4\), the exact sequence doesn't split and we do not have the advantage of determining \(RH_{m,2s}\) as a tensor product. Instead, we encounter the semi-direct product \(H_{m,2s} \cong C_2 \rtimes Spin(m-2s)\) and to continue with the study of \(RH_{m,2s}\), we have to analyse those representations induced by the semi-direct product. However, rather than taking this approach, we choose to overcome the disadvantage by constructing the generators of \(RH_{m,2s}\) as representations derived from that of the maximal torus of \(Spin(m)\). (The methodology is explained at the end of this section.) This approach also facilitates the investigation of the  $RG$-module structure on $RH_{m,2s}$ through the restriction map,  \(\Res\). Furthermore, since \(\omega \notin \tilde{T},\) the standard maximal torus is not a suitable candidate to perform the above mentioned strategy. So, we resort to find a (conjugate) maximal torus  \(T = h^{-1}\tilde{T}h\) that contains \(\omega\)  for some \(h \in Spin(m)\), so that we have the following commutative diagram,
\begin{equation}\label{dprecommutative}
\xymatrix{
\langle{\omega}\rangle \times Spin(m) \ar^{\mu}[r] & H_{m,2s} \subseteq Spin(m)\\
\langle{\omega}\rangle \times T(n-s) \ar^{\mu}[r] \ar^i[u] & T(n) \ar^j[u]
}
\end{equation}
Here, \(i\) and \(j\) are inclusions and \(\mu\) denotes multiplication. 

    The following lemma gives explicit description of the required maximal torus. Its proof is by direct calculation and hence we omit it.

\begin{lemma}
The element $\omega \in T$ where \(T = h^{-1}\tilde{T}h\) with $h= h_1h_2\cdots h_{s/2}$ where
\[
h_i= 1+ e_{4i-3}e_{4i}- e_{4i-2}e_{4i}+ e_{4i-1}e_{4i}.
\]
To be precise, $T$ is given by
\begin{multline*}
T= \{(\cos\theta_1 +v_1v_2 \sin\theta_1)(\cos\theta_2 -\omega_1\omega_2 \sin\theta_2)\cdots (\cos\theta_s- \omega_{s-1}\omega_s\sin\theta_s)\\(\cos\theta_{s+1}+ e_{2s+1}e_{2s+2}\sin\theta_{s+1})\cdots(\cos\theta_n +e_{2n-1}e_{2n} \sin\theta_n)\mid 0 \leq \theta_i \leq 2\pi\},
\end{multline*}
where $v_i= \displaystyle{\frac{1}{\sqrt{2}}(e_{2i-1}+ e_{2i})}$ and $\omega_i = \displaystyle{\frac{1}{\sqrt{2}}(e_{2i-1} - e_{2i})}.$
\end{lemma}

\begin{remark}
(i) We see that, by taking $\theta_{2i-1}= 0$ and $\theta_{2i}= -\pi/2,$ the element $\omega \in T.$ \\
(ii) Choice of $T$ is not unique as it varies with \(h\). One of the ways of  getting an appropriate $h$ is by making the most of the vector space structure of \(Spin(m).\) For instance, \(Spin(4) \) is generated by the basis set,
\[\left\{1, e_1e_2e_3e_4, e_ie_j |i \neq j, 1\leq i,j \leq 4 \right\}.\]

This basis allows us to compute \(h\) as a linear combination of its elements such that \(hw_1w_2h^{-1}= \tilde{t} \in \tilde{T}\). In particular, for \(\tilde{t} = -e_3e_4 \in \tilde{T}\), we get $h= 1+ e_1e_4- e_2e_4+ e_3e_4$.
\end{remark}

\noindent Further, we have a homomorphism \(\tau\) from the standard $n$-torus \(T^n\) to the maximal torus \(T(n)\) given by
\begin{multline*}
\tau(e^{i\theta_1}, \ldots, e^{i\theta_n})= (\cos\theta_1 +v_1v_2 \sin\theta_1)(\cos\theta_2 -\omega_1\omega_2 \sin\theta_2)\cdots\\ \cdots (\cos\theta_s- \omega_{s-1}\omega_s\sin\theta_s)(\cos\theta_{s+1}+ e_{2s+1}e_{2s+2}\sin\theta_{s+1})\cdots(\cos\theta_n +e_{2n-1}e_{2n} \sin\theta_n).
\end{multline*}
Let \( \tilde{\omega}= (\underbrace{1, -i, 1,\ldots,-i}_{s~\text{terms}},\underbrace{1,1\ldots,1}_{n- s~\text{terms}})\). We observe that \(\tau(\tilde{\omega})= \omega. \)
Let us define $\tilde{\Omega} = \langle\tilde{\omega}\rangle $ and $\Omega=\langle{\omega}\rangle $. Through the homomorphism \(\tau\), we extend the commutative diagram (\ref{dprecommutative}) to,
\begin{equation}\label{dcommutative}
\xymatrix{
\tilde\Omega \times Spin(m) \ar^{\mu}[r] & H_{m,2s} \subseteq Spin(m)\\
\tilde\Omega \times T(c) \ar^{\mu}[r] \ar^i[u] & T(n) \ar^j[u]\\
\tilde{\Omega} \ar^{\tau}[u] \times T^c \ar^{\mu}[r] & T^n \ar^{\tau}[u]
}
\end{equation}

Now, using the commutative diagram \ref{dcommutative}, we obtain an explicit description of \(RH_{m,2s}\) by employing a technique adapted from \cite[Page 235]{bd}. Given the short exact sequence \[0 \to \tilde\Omega \to H_{m,2s} \to Spin(m-2s) \to 0,\] the representations of \(H_{m,2s}\) possess the following properties: 
\begin{enumerate}[label=Property \arabic*.]
\item We can identify the representations of \(R\tilde\Omega \otimes RSpin(m-2s)\) with the ones in the image, \(\mu^{\#}(RH_{m,2s})\), by checking for their triviality on \(Ker(\mu)\).
\item  Let \(\rho\) be any arbitrary representation in \(R\tilde\Omega \otimes RSpin(m-2s).\) Let \(r_h \in Aut(\tilde\Omega \times Spin(m-2s))\) denote a right translation by \(h\). \(r_h\) induces a ring homomorphism \(r_h^{\ast} \in Aut(R\tilde\Omega \otimes RSpin(m-2s))\)  defined by \(\rho(g)= \rho(gh).\) For \(h \in Ker(\mu)\), \(r_h^{\ast}\) fixes \(\mu^{\#}(RH_{m,2s})\) elementwise. 
\end{enumerate}

By isolating those representations that satisfy the above two properties, we identify the ring \(RH_{m,2s}\) with the subring \(\mu^{\#}(RH_{m,2s}) \subset R\tilde\Omega \otimes RSpin(m-2s)\). As we move forward, it becomes imperative to address four distinct cases depending on \(m\) and \(s\), each forming a separate section for thorough exploration and analysis.\\

\noindent Before we begin our calculations, we draw the reader's attention to some essential relations used in the evaluation of  \(\Res: RSpin(m) \to RH_{m,2s}\).
\begin{lemma}{\label{simplifications}} 
Let \(y = \theta - 1\).
\begin{enumerate}[label=(\roman*)] 
\item[(1)] \(\theta(\theta+1) = \theta+1\).
\item[(2)] For \(j > 0, y^j = (-2)^{j-1}y\).
\item[(3)]  For \(j > 0,\) \((1 + \theta)^j = 2^{j-1}(1+\theta)\).
\item[(4)] For \(t =\pm 1\), \((1+t)^j = 2^{j-1}(1+t)\).
\item[(5)] For \(t =\pm 1\), \((1+t\theta)^j = 2^{j-1}(1+t\theta).\)
\end{enumerate}
\end{lemma}

Furthermore, it is important to recall that
\[
\tau(\tilde{\omega})= \omega, ~\text{where} ~ \tilde{\omega}= (\underbrace{1, -i, 1,\ldots,-i}_{s~\text{terms}},\underbrace{1,1\ldots,1}_{n- s~\text{terms}}).
\]
 From the commutative diagram \ref{dcommutative}, for any generic element $\xi = (1, \ldots, 1, e^{i\theta_{c+1}}, \ldots, e^{i\theta_{n}}) \in T^c$, where $m-2s= 2c$ or $2c+1$, its image under \(\mu^{\#}\) can directly be inferred as follows:  \\
\begin{lemma}\label{limageofz_i}
\[
\mu^{\#}(z_i)= \begin{cases}
~~1  ~~~~~\text{for} ~ i \leq s ~\text{and}~ i ~\text{odd,}\\
-\phi  ~~~\text{for} ~ i \leq s ~\text{and}~ i ~\text{even,}\\
~~z_i ~~~~~\text{for} ~ i > s.
\end{cases}
\]
\end{lemma}

\section{\(K^{\ast}(FV_{m,2s})\) for \(m \equiv 1 \pmod 2 \) and \(s \equiv 0 \pmod 4\)}\label{smoddsequiv0}
As a first step of our computations, we characterize the generators of the ring \(RH_{m,2s}\) by looking for those representations that possess the properties mentioned in Section \ref{s3.3}.
\subsection{$RH_{m,2s}$}\label{RHsubring}  

In this scenario, it is notable that \(m\equiv 1 \pmod 2\) and \(Ker(\mu) = \{(1,1), (\tilde\omega^2,1)\}\). Consequently,  \(R\tilde\Omega = R\mZ/4\mZ = \mZ[\phi]/(\phi^4=1)\) with \(\phi(\tilde\omega) = i\) and \(RSpin(2c+1)= \mZ[\Lambda^1, \Lambda^2, \cdots, \Lambda^{c-1}, \Delta_c]\). \(R(H_{m,2s}/Spin(2c+1)) = R(\mZ/2\mZ) = \mZ[\theta]/(\theta^2=1)\) with the identification \(\theta \sim \phi^2\).

Let \(\rho\) be an arbitrary element of \(R\tilde\Omega \otimes RSpin(2c+1).\) We first check if \(\rho\) is an element of \(RH_{m,2s}\) by testing its triviality on \(Ker(\mu)\). That is, for instance, if \(\rho = \phi\Lambda^i, \rho((\tilde\omega^2,1)) = \phi(\tilde\omega^2)\otimes\Lambda^i(1)) = -I\), then  \(\phi\Lambda^i \notin R\tilde\Omega \otimes RSpin(2c+1).\) This first check on Property 1 is shown by the first two columns of the table below.

As for the check on Property 2, we look for those elements of \(R\tilde\Omega \otimes RSpin(2c+1)\) that remain fixed by \(r_h^{\ast}\). Let \(g \in \tilde\Omega\) and \(g' \in Spin(2c+1).\) For \(\rho = \phi\Lambda^i\), \(\rho(g,g') = \phi(g) \otimes \Lambda^i(g')\). \(\phi(g) \otimes \Lambda^i(g')\) can be evaluated for the four values of \(g \in \{1, \tilde\omega, \tilde\omega^2, \tilde\omega^3\}\) as in the third column. Now, the ring homomorphism resulting due to the translation of elements of \(\tilde\Omega \otimes Spin(2c+1)\) by \(h \in Ker \mu\), denoted by \(r_h^{\ast}\), is calculated in the fourth column for varying values of \(\rho\) and \(g\). Finally, we deduce by comparing the third and fourth column that \(r_h^\ast (\rho) = \pm \rho\) and \(RH_{m,2s}\) is the ring generated by those representations for which \(r_h^\ast(\rho) = \rho.\)\\
\begin{tabular}{|c|c|c|c|c|c|c|c|c|c|c|}
 \hline
 $\rho$  & $\rho(\tilde\omega^2,1)$ &  \multicolumn{4}{c|}{$\rho(g,g')$}& \multicolumn{4}{c|}{$\rho(\tilde\omega^2g,g')$}& $r^{\ast}_{(\tilde\omega^2,1)}$\\
 \hline
       &     & $g=1$ & $g=\tilde\omega$ & $g=\tilde\omega^2$ & $g=\tilde\omega^3$ & $g=1$ & $g=\tilde\omega$ & $g=\tilde\omega^2$ & $g=\tilde\omega^3$ & \\
 \hline
 $\Lambda^i$ & I  & $\Lambda^i$ & $\Lambda^i$ & $\Lambda^i$ & $\Lambda^i$ & $\Lambda^i$ & $\Lambda^i$ & $\Lambda^i$ & $\Lambda^i$ & $\Lambda^i$\\
 $\Delta_c$ & I &  $\Delta_c$ &  $\Delta_c$ &  $\Delta_c$ &  $\Delta_c$ &  $\Delta_c$ &  $\Delta_c$ &  $\Delta_c$ &  $\Delta_c$ &  $\Delta_c$\\
 $\phi \Lambda^i$ & -I & $\Lambda^i$ & $i\Lambda^i$ & $-\Lambda^i$ & $-i\Lambda^i$ & $-\Lambda^i$ & $-i\Lambda^i$ & $\Lambda^i$ & $i\Lambda^i$ & $-\phi\Lambda^i$\\
 $\phi \Delta_c$ & -I &  $\Delta_c$ &  $i\Delta_c$ &  $-\Delta_c$ &  $-i\Delta_c$ &  $-\Delta_c$ &  $-i\Delta_c$ &  $\Delta_c$ &  $i\Delta_c$ &  $-\phi\Delta_c$\\
 $\phi^2  \Lambda^i$ & I & $\Lambda^i$ & $-\Lambda^i$ & $\Lambda^i$ & $-\Lambda^i$ & $\Lambda^i$ & $-\Lambda^i$ & $\Lambda^i$ & -$\Lambda^i$ & $\phi^2\Lambda^i$\\
 $\phi^2  \Delta_c$ & I &  $\Delta_c$ &  -$\Delta_c$ &  $\Delta_c$ &  $-\Delta_c$ &  $\Delta_c$ &  $-\Delta_c$ &  $\Delta_c$&  -$\Delta_c$ & $\phi^2 \Delta_c$\\
 $\phi^3  \Lambda^i$ & -I & $\Lambda^i$& $-i\Lambda^i$ & $-\Lambda^i$ & $i\Lambda^i$ & $-\Lambda^i$ & $i\Lambda^i$ & $\Lambda^i$ & $-i\Lambda^i$ & $-\phi^3\Lambda^i$\\
 $\phi^3 \Delta_c$ & -I &  $\Delta_c$ &  $-i\Delta_c$ &  $-\Delta_c$ &  $i\Delta_c$ &  $-\Delta_c$ &  $i\Delta_c$ &  $\Delta_c$ &  $-i\Delta_c$ &  $-\phi^3\Delta_c$
 \\
 \hline
\end{tabular}\\

Suppose we denote the k-Pontryagin classes corresponding to the ring \(RH_{m,2s}\) by \(\bar{\Pi}_1, \cdots, \bar{\Pi}_{c-1}\), from the above table we have 
\begin{theorem}\label{tRHmodds0mod4}
\[
RH_{m,2s} = \mZ[\Lambda^1, \cdots, \Lambda^{c-1}, \Delta_c, \theta]/\langle\theta^2-1 \rangle = \mZ[\bar{\Pi}_1, \cdots, \bar{\Pi}_{c-1}, \bar{\Delta}_c, y]/\langle y^2+2y \rangle, 
\]
with  \[ \Delta_c^2 = \sum_{i=0}^c 2^{2(c-i)}\bar{\Pi}_i.\]
    
\end{theorem}

\subsection{The Restriction Map}\label{srest_map}
To calculate \(Tor^{\ast}_{RSpin(m)}(RH_{m,2s}; \mZ)\), it is necessary to  determine the \(RSpin(m)\)-module structure of \(RH_{m,2s}\) through the restriction homomorphism \(Res: RSpin(m) \to RH_{m,2s}\). Leveraging the commutative diagram \ref{dcommutative}, we deduce the following result.

\begin{proposition}\label{restriction map}
For \(m = 2n+1\) and \(s\equiv 0 \pmod4\). \(Res: RSpin(m) = \mathbb{Z}[\Pi_1,\ldots, \Pi_{n-1}, \delta_n] \to RH_{m,2s} = \mZ[\bar{\Pi}_1, \cdots, \bar{\Pi}_{c-1}, \bar{\delta}_c, y]/(y^2+2y)\) is determined in terms of its generators by
\begin{align}
Res(\Pi_i) &= \sum_{j=0}^{c} \binom{s/2}{i-j}(-1)^{i-j-1}2^{2(i-j)-1}y\bar{\Pi}_j,\label{e4.2.1-1}\\
Res(\delta_n) &= 2^{s-1}(y+2)\delta_c + 2^{n-1}y.\label{e4.2.1-2}
\end{align}
\end{proposition}
\begin{proof}
Let us define \(\Pi^{'}[t] = \mu^{\#}\Pi[t].\)
\begin{align*}
   \Pi^{'}[t] &= \mu^{\#}  \left(\prod_{i=1}^n\left(1+t(z_i-z_i^{-1})^2\right) \right)\\
    &= \mu^{\#}  \left(\prod_{i=1 \atop i \equiv 0 (2)}^s\left(1+t(z_i-z_i^{-1})^2\right)\prod_{i=1 \atop i \equiv 1 (2)}^s\left(1+t(z_i-z_i^{-1})^2\right)\prod_{i=s+1}^n\left(1+t(z_i-z_i^{-1})^2\right) \right)\\
    &=\left(1+t(\phi^3 -\phi)^2\right)^{s/2}\left(\bar{\Pi}[t]\right) \quad \quad \quad \quad (\text{From Lemma}~ \ref{limageofz_i}).       
 \end{align*}
Now, we see that $(\phi^3- \phi)^2= \theta(\theta- 1)^2= \theta(2-2\theta)= 2y.$ Hence, we have
\[
\Pi^{'}[t]= \left(1+2ty\right)^{s/2}\left(\bar{\Pi}[t]\right).
\]
 
 \noindent Equating the coefficients of \(t^i\) in
    \begin{equation}\label{Resfunction}
     \Pi^{'}[t] =\left(1+2ty\right)^{s/2}\left(\bar{\Pi}[t]\right)
    \end{equation} we get,
\begin{align*}
Res(\Pi_i)= \Pi^{'}_i &= \sum_{J=0}^i \binom{s/2}{J}(2y)^J \bar\Pi_{i-J}   \\
&= \sum_{j=0}^c \binom{s/2}{i-j}(2y)^{i-j}\bar\Pi_{j} \quad \quad \quad (\text{Since} \,  \bar\Pi_i = 0 \,  \text{for}\, i > c)\\
&= \sum_{j=0}^{c} \binom{s/2}{i-j}(-1)^{i-j-1}2^{2(i-j)-1}y\bar{\Pi}_j. \quad\quad \quad (\text{From Lemma \ref{simplifications}-}(ii))
\end{align*}
While calculating $Res(\delta_n),$ by Theorem \ref{trspin}(1), we recall that $\delta_n= \Delta_n- 2^n.$ Now, 
\begin{align*}
 Res(\Delta_n)= \mu^{\#}(\Delta_n) &= \mu^{\#}  \left(\prod_{i=1}^n(z_i+z_i^{-1}) \right)\\
   &= j^{\#}  \left(\prod_{i=1 \atop i \equiv 0 (2)}^s\left(z_i+z_i^{-1}\right)\prod_{i=1 \atop i \equiv 1 (2)}^s\left(z_i+z_i^{-1}\right)\prod_{i=s+1}^n\left(z_i+z_i^{-1}\right) \right)\\
   &= 2^{s/2}(\phi +\phi^3)^{s/2}\Delta_c
\end{align*}
We know that $s$ is even and hence $(\phi +\phi^3)^{s/2}= (\phi +\phi^3)^{2(s/4)}= \theta^{s/4}(\theta+ 1)^{s/2}.$ This gives us
\begin{align*}
\mu^{\#}(\Delta_n)&=    \begin{cases}
   ~~ 2^{s/2}\theta\left(\theta+1\right)^{s/2}\Delta_c  \quad \quad ~~~~~\text{for} ~ s/4 \equiv 1 \pmod 2\\
   ~~ 2^{s/2}\left(\theta+1\right)^{s/2}\Delta_c \quad \quad ~~~~\text{for} ~ s/4 \equiv 0 \pmod 2
\end{cases}\\
&= 2^{s-1}(\theta+1)\Delta_c.  \quad \quad \quad (\text{From Lemma } \ref{simplifications}-(3))
\end{align*}
Hence, we get 
\[
Res(\delta_n)= 2^{s-1}(y+2)(\delta_c+ 2^c)- 2^n= 2^{s-1}y\delta_c+ 2^{n-1}y+ 2^s\delta_c+ 2^n- 2^n.
\]
Simplification gives us desired expression for \(Res(\delta_n).\)
\end{proof}

\subsection{$Tor^{\ast}_{RSpin(m)}(RH_{m,2s}; \mZ)$}\label{Torcomputation}
We proced to compute \(Tor^{\ast}_{RSpin(m)}(RH_{m,2s}; \mZ)\). Theoretically, the Koszul resolution \cite[Chapter XXI, \S4] {L02} allows us to construct a corresponding chain complex whose homology gives our desired invariant. However, in practice,  it is almost impossible to compute the homology of the resulting complex. Nevertheless, by establishing relations among the generators of the rings \(RSpin(m)\) and \(RH_{m,2s}\), we invoke the Change of Rings Theorem \cite[Chapter XVI, Theorem 6.1]{ce}, \cite[Lemma 4.5]{g}  to obtain a simplified expression for \(Tor^{\ast}_{RSpin(m)}(RH_{m,2s}; \mZ),\) thereby reducing the complexity of the calculations involved. 

It is imperative to state the few properties of the Koszul resolution \cite[Section 6.3]{bh} which we use  to obtain the computable equivalent complex \\

Let \(\epsilon_{RG}: RG \to R\{1\} = \mZ\) be the augmentation map that maps every (virtual) representation to its (virtual) dimension. Similarly, we define the augmentation map \(\epsilon_{RH}: RH \to \mZ\).
Let \(RG= \mZ[\gamma_1, \ldots , \gamma_n]\) be such that \(\epsilon_{RG}(\gamma_i) = 0\), for \(1 \leq i \leq n\) and let \(RH\) be generated by \(h_1, \ldots, h_m\) with \(\epsilon_{RH}(h_i) = 0 \), for \(1 \leq i \leq m\).
\begin{enumerate}[label=Koz. \arabic*.]

\item   If \(RH\) is a free \(RG\)-module, the above chain complex becomes an exact sequence and
\[
Tor^{RG}_k(RH, \mZ) =  \begin{cases}
RH \otimes_{RG} \mZ   ~~~~~\text{for} ~k = 0 \\
0  \quad \quad ~~~~~\text{for} ~k \neq 0.
\end{cases}
\] 
\item If  for all \( i\), \(Res(\gamma_i) = 0\), that is, if \(RH\) is a trivial \(RG\)-module, then \[ Tor_{RG}^{\ast}(RH, \mZ) \cong \Lambda_{\mZ}^{\ast}[u_1, \ldots, u_n].\]
\item Let us define a homomorphism \(\overline{Res}: \overline{RG} = \mZ[\gamma_1, \cdots , \gamma_m] \to RH\) by \(\overline{Res}(\gamma_i) = h_i\) for \( 1 \leq i \leq m\),  so that \(RH\) becomes a \(\overline{RG}\)-module. If \(RH\) is a free (or more generally a flat) \(\overline{RG}\)-module, it follows that \[ Tor^{\ast}_{RG}(RH, \mZ) = Tor^{\ast}_A(B, \mZ),\] where \(A =RG/\langle\gamma_1, \ldots, \gamma_m\rangle\) and \(B = RH/\langle h_1, \ldots h_m\rangle.\) Further, if \(\tau_1, \ldots, \tau_s \in A\) and \(\overline{Res}(\tau_i) = 0\) for \(1 \leq i \leq s\), then setting \(A_1 = A/(\tau_1, \cdots, \tau_s)\), \(B\) is an \(A_{1}\)-module and
\[ Tor^{\ast}_A(B, \mZ) \cong \Lambda^{\ast}_{\mZ}[t_1, \cdots, t_s] \otimes Tor^{\ast}_{A_1}(B, \mZ).\] 
\end{enumerate}

Now, we analyse the structure of \(RH_{m,2s}.\)
\begin{lemma}\label{lrelationspibarandpi}
Let $\Gamma$ be the subring of $RH_{m,2s}$ generated by $\Pi^{'}_1, \ldots, \Pi^{'}_c.$ Then $RH_{m,2s}$ is a free $\Gamma$ module with generators $y, \delta_c.$ 
\end{lemma}

\begin{proof}
From \ref{Resfunction}, we have
\begin{align*}
\bar{\Pi}[t] &= (1+2ty)^{-s/2}\Pi^{'}[t]\\
\sum_{j\geq 0}\bar\Pi_j t^j &= \left[\sum_{j\geq 0}\binom{s/2+j-1}{j}(-1)^j(2ty)^j\right]\left[\sum_{j\geq 0}\Pi_j^{'}t^j\right]\\
&= \sum_{i\geq 0}\sum_{j =0}^{i} (-1)^j\binom{s/2 + j - 1}{j}(2y)^j\Pi^{'}_{i-j}t^i.
\end{align*}

\noindent By equating the coefficients of \(t^i\) we have,
\begin{equation}\label{eq_relation_bar-Pi}
\bar{\Pi}_i = \sum_{j =0}^{i} (-1)^j\binom{s/2 + j - 1}{j}(2y)^j\Pi^{'}_{i-j}
\end{equation}
In particular, the matrix expressing the elements \(\bar{\Pi}_i\) in terms of \(\Pi^{'}_i\) is triangular with diagonal entries as 1.

To conclude the proof, define $\phi: RH_{m,2s} \to RH_{m,2s}$ by 
\[
\phi(y)= y, ~~\phi(\delta_c)= \delta_c, ~~\phi(\bar{\Pi}_i)= \Pi^\prime_i ~~\text{for}~ 1 \leq i \leq c.
\]
The map $\phi$ is an isomorphism and it intertwines $\Gamma$-module structure on $RH_{m,2s}$ and $\mZ[\Pi_1^\prime, \ldots, \Pi_c^\prime]$-module structure on $RH_{m,2s}.$ 

This concludes the proof.
\end{proof}

\noindent By applying the property \textbf{Koz. 3} discussed earlier in this section, we have

\begin{proposition}\label{p4.3.2}
\[Tor^{\ast}_{RSpin(m)}(RH_{m,2s}; \mZ) \approx Tor^{\ast}_A(B;\mZ),\]
where
\[
A= RSpin(m)/\langle \Pi_1, \cdots, \Pi_c\rangle ~~~\text{and}~~~  B = RH/\langle \Pi_1^{'}, \cdots, \Pi^{'}_c\rangle.
\]
\end{proposition}
An immediate implication of Proposition \ref{p4.3.2} is  the relations in \(B\) exhibited in Lemma \ref{generatorsofB} and Corollary \ref{crelationsinB}.
\begin{lemma} \label{generatorsofB} In B, 
\begin{align*}
\bar{\Pi}_i &= -\binom{(s/2)+i-1}{i}2^{2i-1}y, \quad \quad \text{for} \,\, \, \, \,  1 \leq i \leq c\,;\\
\Pi^{'}_i &= \binom{(s/2) + i -1}{i}2^{2i-1}y - \sum_{j=c+1}^{i-1} \binom{(s/2)+i-j-1}{i-j}2^{2(i-j)}\Pi_j^{'}, \quad \quad \text{for} \,\, \, \, \, c+1 \leq i \leq n.
\end{align*}
 \end{lemma}
 \begin{proof}
 In B, the only \(\Pi_i^{'}\)s that survive in Equation \ref{eq_relation_bar-Pi} are  \(\Pi_i^{'}\) for \(i = 0 , c+1, \cdots , n\).\\So for \(1 \leq i \leq c\), we have \(\bar{\Pi}\) as mentioned.\\ 
 For \(c+1 \leq i \leq n\), \(\bar{\Pi}_i = 0\). Hence, we have
\begin{multline*}
0 = \binom{(s/2)+i-1}{i}(-2y)^i+\sum_{J =0}^{i-(c+1)} (-1)^j\binom{(s/2) + J - 1}{J}(2y)^J\Pi^{'}_{i-J}\\
+ \sum_{J = i-c}^{i-1} (-1)^J\binom{(s/2) + J - 1}{J}(2y)^J\Pi^{'}_{i-J}
\end{multline*}
\[
0= (-1)^{i}\binom{(s/2)+i-1}{i}(2y)^i + \sum_{j =c+1}^{i} (-1)^{i-j}\binom{(s/2) + i-j - 1}{i-j}(2y)^{i-j}\Pi^{'}_{j} .
\]
By Lemma \ref{simplifications}(2)), we have
\[
0 = (-1)^{2i-1}\binom{(s/2)+i-1}{i}2^{2i-1}y + \Pi_i^{'} - \sum_{j =c+1}^{i-1} \binom{(s/2) + i-j - 1}{i-j}2^{2(i-j)-1}y\Pi^{'}_{j}.
\]
So, $\displaystyle{\Pi_i^{'}= \binom{(s/2)+i-1}{i}2^{2i-1}y - \sum_{j =c+1}^{i-1} \binom{(s/2) + i-j - 1}{i-j}2^{2(i-j)}\Pi^{'}_{j}}$ as \(y\Pi_i^{'} = -2\Pi_i^{'}\) 
 \end{proof}
On direct evaluation for \(\bar{\Pi}_i\) in relations obtained in Theorem \ref{tRHmodds0mod4}, we obtain 
 \begin{corollary}\label{crelationsinB}
 In $B,$ in addition to the relation \(y^2 + 2y = 0\), we have
 \[
 \delta_c^2 + 2^{c+1}\delta_c = 2^{2c-1}y \bigg[1-\binom{(s/2)+c}{c}\bigg].
 \]
\end{corollary} 
  
To conclude,
\begin{proposition}
As a ring, \(B \cong \mZ[\delta_c, y]/I,\) where \(I = \langle y^2+2y, \delta_c^2 + 2^{c+1}\delta_c - 2^{2c-1}y \left[1-\binom{(s/2)+c}{c}\right]\rangle\).
\end{proposition}

\noindent Now, we explore the ring $A$. Here, we reassign most of the generators of \(A\) with the ones in the kernel of \(\Res:RSpin(m) \to RH_{m,2s}\). Because this strategic reassignment ensures that the \(Tor^\ast\) of the associated Koszul resolution is equivalent to the Koszul resolution corresponding to \(Tor^{\ast}_{RG}(RH; \mZ)\) and is computable. Thus, the determination of $K^{\ast}(FV_{m,2s})$ becomes practically feasible.

\begin{lemma}\label{l4.3.6}
There exist $P_{c+1}, \ldots, P_{n-1}$ in $A$ such that there is an isomorphism between $A= \mZ[\Pi_{c+1}, \ldots, \Pi_{n-1}, \delta_n]$ and $\mZ[P_{c+1}, \ldots, P_{n-1}, \delta_n]$ and 
\[
\Res(P_i)= \binom{(s/2) + i -1}{i}2^{2i-1}y.
\]
\end{lemma}
\begin{proof}
Define
\[
P_i:= \Pi^{'}_i + \sum_{j=c+1}^{i-1} \binom{(s/2)+i-j-1}{i-j}2^{2(i-j)}\Pi_j^{'}
\] 
for $i= c+1, c+2, \ldots, n-1$ and consider
\begin{align*}
\Phi  \colon A & \to \mZ[P_{c+1}, \ldots, P_{n-1}, \delta_n]\\
\Pi_i & \mapsto P_i\\
\delta_n &\mapsto \delta_n
\end{align*}
Clearly the matrix of transformation is a triangular matrix with each diagonal entry as $1$ and hence $\Phi$ is an isomorphism. Explicit description of $\Res(P_i)$ is obtained using by Lemma \ref{generatorsofB}.
\end{proof}

The following problem, from the book Algebra by Lang \cite[Page 547, Exercise 25]{L02}, plays a decisive role in getting new generators. We briefly describe its solution.
    
\begin{problem}\label{pGCDMatrix}
Consider $(b_1, b_2, \ldots, b_n)$ where $b_i \in \mathbb{Z}$ and 
\[
GCD(b_1, b_2, \ldots, b_n)= 1
\]
then $(b_1, b_2, \ldots, b_n)$ can be expressed as a first row of an invertible matrix in $GL_n(\mZ).$
\end{problem}
\begin{proof}
We want to show that $\begin{bmatrix}b_1 &  b_2 & \cdots & b_n\end{bmatrix}= e_1M,$ where $e_1= \begin{bmatrix}1 & 0 & 0 & \cdots & 0\end{bmatrix}.$\\
Equivalently, we want to show that 
\[
\begin{bmatrix}b_1 &  b_2 & \cdots & b_n\end{bmatrix}N= e_1,
\]
for some non-singular matrix $N.$ We get desired matrix $N$ by performing row operations on $\begin{bmatrix}b_1 &  b_2 & \ldots & b_n\end{bmatrix}$ by matrices from $GL_n(\mZ).$\\

\noindent We establish the result for $\begin{bmatrix}b_1 &  b_2 & b_3\end{bmatrix}$ and similar argument can be imitated for a longer row.

Let GCD$(b_2, b_3)= \alpha.$ We have $b_2x_2+ b_3x_3= \alpha.$ Hence, $a_2x_2+ a_3x_3= 1,$ where $a_2= b_2/\alpha$ and $a_3= b_3/\alpha.$ Now,
\[
\begin{bmatrix}b_1 & b_2 & b_3 \end{bmatrix} \begin{bmatrix} 1 & 0 & 0\\ 0 & x_2 & a_3\\ 0 & x_3 & a_2 \end{bmatrix}= \begin{bmatrix} b_1 & \alpha & b_2a_3+ b_3a_2\end{bmatrix}.
\]
We know that $\alpha$ divides $b_2a_3+ b_3a_2.$ Let $b_2a_3+ b_3a_2= \alpha y.$ We have 
\[
\begin{bmatrix}b_1 & \alpha & b_2a_3+ b_3a_2 \end{bmatrix} \begin{bmatrix} 1 & 0 & 0\\ 0 & 1 & -y\\ 0 & 0 & 1 \end{bmatrix}= \begin{bmatrix} b_1 & \alpha & 0\end{bmatrix}.
\]
Given that GCD$(b_1, b_2, b_3)= 1.$ That is, GCD$(b_1, \alpha)= 1$ and hence there are integers $z_1, z_2$ so that $b_1z_1+ \alpha z_2= 1.$ Now,
\[
\begin{bmatrix}b_1 & \alpha & 0 \end{bmatrix} \begin{bmatrix} z_1 & -\alpha & 0\\ z_2 & b_1 & 0\\ 0 & 0 & 1 \end{bmatrix}= \begin{bmatrix} 1 & 0 & 0\end{bmatrix}.
\]
Therefore, $\begin{bmatrix}b_1 &  b_2 & b_3\end{bmatrix}N= e_1$ where 
\[
N= \begin{bmatrix} 1 & 0 & 0\\ 0 & x_2 & a_3\\ 0 & x_3 & a_2 \end{bmatrix}
\begin{bmatrix} 1 & 0 & 0\\ 0 & 1 & -y\\ 0 & 0 & 1 \end{bmatrix}
\begin{bmatrix} z_1 & -\alpha & 0\\ z_2 & b_1 & 0\\ 0 & 0 & 1 \end{bmatrix}.
\]
Clearly $N$ is invertible since each of the matrices on the right hand side is invertible.
\end{proof}

Moving forward, we fix a few notations.
\begin{itemize}
\item Let $b_0$= GCD $\displaystyle{\bigg\{\binom{(s/2) + i -1}{i}2^{2i-1}\bigg|i= c+1, c+2, \ldots, n-1\bigg\}}.$
\item Let $\beta_{c+1}, \beta_{c+2}, \ldots, \beta_{n-1} \in \mZ$ such that $\displaystyle{b_0= \sum_{i= c+1}^{n-1}\beta_i \binom{(s/2) + i -1}{i}2^{2i-1}}.$
\end{itemize}

We observe that GCD$(\beta_{c+1}, \beta_{c+2}, \ldots, \beta_{n-1})= 1.$ We use \ref{pGCDMatrix} and have the following Lemma.

\begin{lemma}\label{l4.3.8}
There exist $U_{1}, U_{2}, \ldots, U_{s-1}$ in $A$ such that there is an isomorphism $\Psi: A= \mZ[P_{c+1}, P_{c+2}, \ldots, P_{n-1}, \delta_n] \to \mZ[U_1, U_2, \ldots, U_{s-1}, \delta_n]$ with 
\[
\Res(U_1)= b_0y, \,\,\,\, \Res(U_l)= a_l b_0y
\]
where $l= 2, \ldots, s-1$ and $\alpha_l \in \mZ.$
\end{lemma}
\begin{proof}
Using Corollory \ref{pGCDMatrix}, we can find a matrix $E= [e_{i,j}] \in GL_{s-1}(\mZ)$ with $e_{1,j}= \beta_{c+j}$ for $1 \leq j \leq s-1.$ Define
\[
U_l:= \sum_{j=1}^{s-1} e_{l,j}P_{c+j} \,\,\,\,\, l= 1, 2, \ldots, s-1.
\]
The map $\Psi$ is given by
\begin{align*}
\Psi: A & \to \mZ[U_1, U_2, \ldots, U_{s-1}, \delta_n],\\
P_{c+l} &\mapsto U_l,\\
\delta_n & \mapsto \delta_n.
\end{align*}
Then the matrix of transformation of $\Psi$ is $\diag(E, 1)$ which is invertible hence we conclude that $\Psi$ is an isomorphism. Direct calculation shows us that 
\begin{align*}
\Res(U_1) &= \sum_{j=1}^{s-1} e_{1,j}\Res(P_{c+j})= \sum_{i= c+1}^{n-1}\beta_i \binom{(s/2) + i -1}{i}2^{2i-1}y= b_0y,\\
\Res(U_l) &= \sum_{j=1}^{s-1} e_{l,j}\Res(P_{c+j})= \sum_{i= c+1}^{n-1}e_{l, j-c} \binom{(s/2) + i -1}{i}2^{2i-1}y= a_l b_0y.
\end{align*}
\end{proof}

This allows us to get a set of generators for $A$, most of  which are mapped to zero under the restriction map \(\Res.\) 

\begin{lemma}\label{l4.3.9}
There exist $V_1, V_2, \ldots, V_{s-1} \in A$ such that there is an isomorphism $\Theta \colon A= \mZ[U_1, U_2, \ldots, U_{s-1}, \delta_n] \to \mZ[V_1, V_2, \ldots, V_{s-1}, \delta_n]$ with 
\[
\Res(V_1)= b_0y, \,\,\,\,\, \Res(V_l)= 0
\]
for all $l= 2, \ldots, s-1.$
\end{lemma}
\begin{proof}
Define 
\begin{align*}
V_1 &:= U_1,\\
V_l &:= U_l- a_l U_1 ~~\text{for}~ l= 2, \ldots, s-1.
\end{align*}
Now the map $\Theta$ defined by
\begin{align*}
\Theta \colon A &\to \mZ[V_1, V_2, \ldots, V_{s-1}, \delta_n],\\
U_l &\mapsto V_l,\\
\delta_n &\mapsto \delta_n.
\end{align*}
The matrix of transformation of $\Theta$ is a triangular matrix with each of its diagonal entry as \(1\). Hence $\Theta$ is an isomorphism. A direct checking confirms our claim about the restriction map $\Res$.
\end{proof}
To consolidate, using Lemmas \ref{l4.3.6}, \ref{l4.3.8} and \ref{l4.3.9}, we have reassigned the generators  \(\Pi_{c+1}, \ldots, \Pi_{n-1}\) of \(A\) to \(V_1, V_2, \ldots, V_{s-1}\). Moreover, the submodule \(\mZ[V_2, \ldots, V_{s-1}]\) of \(A\) has a trivial \(A\)-module structure. Using the properties \textbf{Koz. (2)} and \textbf{Koz. (3)}, we deduce
\begin{proposition}\label{p4.3.10}
    \[Tor^{\ast}_A(B;\mZ) \cong \Lambda^{*}_{\mZ}[t_1, \cdots, t_{s-2}] \otimes_{\mZ} Tor^{\ast}_{A'}(B; \mZ) \]
    with \(dim \, t_i = 1\) and \(A' = \mZ[V_1, \delta_n]\) 
\end{proposition}

The above proposition implies that it is enough to find homology of the Koszul complex given by
\[ 
0 \xrightarrow{d_3} B(x_1\wedge x_2) \xrightarrow{d_2} Bx_1 \oplus Bx_2 \xrightarrow {d_1} B \rightarrow 0 
\]
Recall,    \(B \cong \mZ[\delta_c, y]/ I,\) where 
$I = \bigg\langle y^2+2y, \delta_c^2 + 2^{c+1}\delta_c - 2^{2c-1}y \left[1-\binom{(s/2)+c}{c}\right] \bigg\rangle.$

  Here,

 \begin{align*}
d_1(x_1) & := Res(V_1)= b_0y, ~~  ~ \text{where}~ b_0= GCD\left\{2^{2i-1}\binom{(s/2) + i -1}{i}\bigg | c+1 \leq i \leq n-1 \right\}. \\
d_1(x_2)&:= Res(\delta_n) =  2^{s-1}(2+y)\delta_c +2^{n-1}y  \\ 
d_2(x_1 \wedge x_2) &= b_0yx_2 - (2^{s-1}(2+y)\delta_c + 2^{n-1}y)x_1.
\end{align*}

\noindent ${\bf\underline{H_0}:}$
\begin{align*}
Ker(d_0) &=  B\\
Im(d_1) &= \langle  b_0y , 2^{s-1}(2+y)\delta_c +2^{n-1}y\rangle\\
H_0 &= B/ \langle b_0y,   2^{s-1}(2+y)\delta_c +2^{n-1}y \rangle B
\end{align*}

\noindent Note that \(y(2^{s-1}(2+y)\delta_c +2^{n-1}y) = -2^ny = 0\).\\
 Suppose \(GCD\{2^n, b_0\} = 2^{\alpha} \), then the relation  \(b_0y\) can be replaced by \(2^{\alpha}y\) in the K-ring expression.

\noindent ${\bf \underline{H_2}:}$

\begin{align*}
Im(d_3) &= 0\\
Ker(d_2) &= \bigg\{p(x_1 \wedge x_2) \bigg| d_2(p(x_1 \wedge x_2)) = 0 , p  \in B \bigg\} 
\end{align*}
Any \(p \in B\), can be written as \(p = p_1\delta_cy + p_2\delta_c + p_3y +p_4\).\\
By solving relation $d_2(p(x_1 \wedge x_2)) = 0,$ we get
\[
p_4= 2p_3, ~~~p_2= 2p_1, ~~~p_4= 2^{c+1}p_2.
\]
Hence, $p= (\delta_cy + 2\delta_c+ 2^{c+1}y+ 2^{c+2})= (2+y)(2^{c+1}+ \delta_c)$. \\
Let \(v= (2+y)(2^{c+1}+\delta_c)x_1 \wedge x_2 \in (x_1 \wedge x_2)B\), then
\[
Ker(d_2)=  v\mZ.
\]

\noindent ${\bf \underline{H_1}:}$
\begin{align*}
Im(d_2) &= \langle d_2(x_1 \wedge x_2) =  b_0yx_2 - (2^{s-1}(2+y)\delta_c + 2^{n-1}y)x_1\rangle B.\\
Ker(d_1) &= \bigg\{px_1+ qx_2 \bigg| pd_1(x_1)+ qd_1(x_2) =0; p,q \in B\bigg\}
\end{align*}
Let \( p = p_1\delta_cy + p_2\delta_c + p_3y +p_4\) and \( q = q_1\delta_cy + q_2\delta_c + q_3y + q_4\) for  \( p_1, \ldots, p_4 , q_1, \ldots q_4 \in \mZ\).\\
\(pd_1(x_1) +  q(d_1(x_2)) =0\) gives us the following system of equations:
\begin{align*}
q_4- 2^{c+1}q_2 &= 0\\ 
b_0p_2- 2p_1b_0+  2^{n-1}q_2-  2^nq_1 &=  0\\
-2b_0p_3+ p_4b_0- 2^nq_3+ 2^{n+c}q_2 &=  0 
\end{align*}

The solutions to the above system  yield the generators of \(Ker(d_1)\) which are discussed here. We discuss a  systematic way of choosing the generators  using Grobner basis in  Appendix A.
 
\begin{lemma}\label{lkergenerators}
The generators of $ker(d_1)$ are $\begin{cases} u_1, u_2, u_3 \qquad\quad\,\,\,  \text{when}\,\,\, \alpha= n\\ u_1, u_2, u_3, u_4 \qquad \text{when} \,\,\,\alpha< n \end{cases}$ where
\begin{align*}
u_1 &= (y+2)x_1,\\
u_2 &= (y+2) (2^{c+1}+ \delta_c) x_2,\\
u_3 &= 2^{n- \alpha}x_1+ 2^{-\alpha}b_0yx_2,\\
u_4 &= 2^{n-1-\alpha}\delta_cx_1 - 2^{-\alpha}b_0(\delta_c +2^cy + 2^{c+1})x_2.
\end{align*}
\end{lemma}
\begin{proof} Let $a \in ker(d_1).$ Then, $a= (p_1\delta_c y+ p_2\delta_c+ p_3 y+ p_4)x_1+ (q_1\delta_c y+ q_2\delta_c+ q_3 y+ q_4)x_2$ where $p_i, q_j \in \mathbb{Z}$ for $1 \leq i, j \leq 4.$ 

 Let $g_1:= GCD(b_0, 2^{n-1}).$ We get $\gamma_1 \in \mathbb{Z}$ such that 
\[
p_2- 2p_1= \gamma_1 (g_1^{-1}2^{n-1}) \,\,\text{and} \,\, 2q_1- q_2= \gamma_1 (g_1^{-1}b_0).
\] 
Let $g_2:= 2^\alpha= GCD(b_0, 2^n).$ We get $\gamma_2 \in \mathbb{Z}$ such that
\[
p_4- 2p_3= \gamma_2 (g_2^{-1}2^n) \,\,\text{and} \,\, q_3- 2^cq_2= \gamma_2 (g_2^{-1}b_0).
\]
\begin{enumerate}
\item If $\alpha= n,$ we have
$a= (p_1\delta_c+ p_3)u_1+ q_1u_2+ \gamma_2u_3+ \gamma_1(\delta_cu_3- 2^{-n}b_0u_2).$
\item If $\alpha \leq n-1,$ we get that 
$a= (p_1\delta_c+ p_3)u_1+ q_1u_2+ \gamma_2u_3+ \gamma_1u_4.$
\end{enumerate}
\end{proof}

\noindent Based on the preceding arguments, we have
\[Tor^{\ast}_{RSpin(m)}(RH_{m,2s}; \mZ) \cong \Lambda^{\ast}_{\mZ}[t_1, t_2, \ldots t_{s-2}, u_1, u_2, u_3,u_4,v] \otimes \mZ[y,\delta_c]/I\]
where $I$ has all the relations that hold in \(B, H_0, H_1\) and \(H_2.\) which we describe now.\\ 

\noindent {\bf Multiplication in \(Tor^{\ast}_{RSpin(m)}(RH_m,2s; \mZ)\)}: We know that if $A$ and $B$ are $R$-algebras, we have a multiplication
\[
\mu\colon Tor_R^m(A, B) \times_R Tor_R^n(A, B) \to Tor_R^{m+n}(A, B)
\]
which makes $Tor_R^\ast(A, B)$ into a graded associated $R$-algebra that is 'graded-commutative'. See \cite[Page 640, Exercise A3.20]{e}. In our case, multiplication turns out to be wedge product modulo relations obtained above.\\

\noindent {\bf Relations in $I$}: There are some immediate relations in $Tor^\ast,$ which are obtained by observations such as $(y+2)u_3= 2^{n-\alpha}u_1, u_2u_4= u_1u_3= 0$. For remaining relations, we consider various linear combinations and equate then to zero. Relations in $I$ are described in the following table. 
\[
\begin{tabular}{| m{1cm} | m{11cm} |}
\hline
\(B\) &  \begin{itemize} 
\item \(y^2+2y = 0 ,\)
\item $ \delta_c^2+ 2^{c+1}\delta_c- 2^{2c-1}y \bigg[1- \binom{(s/2)+c}{c} \bigg]= 0.$
\end{itemize} \\
\hline
\(H_0\)  & \begin{itemize} 
\item \(2^{\alpha}y = 0,\)
\item $2^{s-1}(2+y)\delta_c+ 2^{n-1}y= 0.$ 
\end{itemize}\\ 
\hline 
$H_1$ & \begin{itemize} 
\item $2^{\alpha}u_3 - (2^{s-1} \delta_c + 2^{n-1})u_1$
\item $(y+2)u_3= 2^{n-\alpha}u_1,$
\item $yu_1= yu_2= \delta_cu_2= 0,$
\item $(y+2)u_4= 2^{n-1-\alpha}\delta_cu_1- 2^{-\alpha}b_0u_2,$
\item $(y+2)(2^c\delta_c+1)u_3 +(y+2)\delta_cu_4=2^{n-\alpha}u_1 $,
\item $2^{c-\alpha+1}b_0u_2 + (y+2)(2^{c+1} +\delta_c)u_4 = 0$,
\item $2^{c-\alpha -1}b_0u_2 + (2^{c-1}\delta_c + 2^{2c-2}y\bigg[1- \binom{(s/2)+c}{c} \bigg] +2^{c-1}\delta_cy)u_3 + (2^c+\delta_c+2^cy+\delta_cy)u_4=0.$
\end{itemize} \\
\hline
$H_2$ & \begin{itemize}
\item $u_2u_4= u_1u_3= \delta_cu_1u_4= 0,$
\item $2^{n-\alpha}b_0u_1u_2+2^nu_1u_4+b_0u_2u_3 =0,$
\item $2^{n-2\alpha}b_0u_1u_2+2^{-\alpha}b_0u_2u_3 +2u_3u_4 = 0,$
\item $-2^{n-\alpha+1}u_1u_4 + 2^{-\alpha}b_0u_2u_3+2u_3u_4 = 0,$
\item $2^{n-2\alpha}b_0u_1u_2 + 2^{n-\alpha}u_1u_4+2u_3u_4=0.$
\end{itemize}\\
\hline 
\end{tabular}\\
\]

Now, since generators of $Tor^\ast$ have degree zero or one, we conclude that $K^{\ast}FV_{m,2s} \cong Tor^{\ast}_{RG}(RH, \mZ).$ We refer to \cite[Page 39-40, Section 6]{aguz} for a detailed relation between $E^\infty$ term and $K^\ast$ ring. Following result is a consequence of  \cite[Page 40, Corollary 6.5]{aguz}.

\begin{theorem}
For \(m \equiv 1 \pmod 2 \) and \(s \equiv 0 \pmod 4\), let $2^{\alpha}= GCD\{2^n, b_0\}$ with $b_0$= $GCD\displaystyle{\bigg\{\binom{(s/2) + i -1}{i}2^{2i-1}\bigg|i= c+1, c+2, \ldots, n-1\bigg\}}$ and 
\[
K^\ast(FV_{m,2s}) \cong \begin{cases}
\Lambda^{\ast}[t_1, t_2, \ldots t_{s-2}, u_1, u_2, u_3,v] \otimes \mZ[y,\delta_c]/I \hspace{0.9cm}\text{when} ~\alpha= n,\\
\Lambda^{\ast}[t_1, t_2, \ldots t_{s-2}, u_1, u_2, u_3,u_4,v] \otimes \mZ[y,\delta_c]/I  ~~~\text{when} ~\alpha< n,
\end{cases}
\]
where \(I\) is generated by $y^2+2y, ~\delta_c^2+ 2^{c+1}\delta_c- 2^{2c-1}y \bigg[1- \binom{(s/2)+c}{c} \bigg], ~2^{\alpha}y, ~2^{s-1}(2+y)\delta_c+ 2^{n-1}y, ~2^{\alpha}u_3 - (2^{s-1} \delta_c + 2^{n-1})u_1, ~(y+2)u_3- 2^{n-\alpha}u_1, ~yu_1, ~yu_2, ~\delta_cu_2, ~(y+2)u_42 2^{n-1-\alpha}\delta_cu_1- 2^{-\alpha}b_0u_2, ~(y+2)(2^c\delta_c+1)u_3 +(y+2)\delta_cu_4-2^{n-\alpha}u_1, ~2^{c-\alpha+1}b_0u_2 + (y+2)(2^{c+1} +\delta_c)u_4, ~2^{c-\alpha -1}b_0u_2 + (2^{c-1}\delta_c + 2^{2c-2}y\bigg[1- \binom{(s/2)+c}{c} \bigg] +2^{c-1}\delta_cy)u_3 + (2^c+\delta_c+2^cy+\delta_cy)u_4, ~u_2u_4, ~u_1u_3, ~\delta_cu_1u_4, ~2^{n\alpha}b_0u_1u_2+2^nu_1u_4+b_0u_2u_3, ~2^{n-2\alpha}b_0u_1u_2+2^{-\alpha}b_0u_2u_3 +2u_3u_4, ~-2^{n-\alpha+1}u_1u_4 + 2^{-\alpha}b_0u_2u_3+2u_3u_4, ~2^{n-2\alpha}b_0u_1u_2 + 2^{n-\alpha}u_1u_4+2u_3u_4$ and $u_1u_2- 2v.$
\end{theorem}

\section{\(K^{\ast}(FV_{m,2s})\) for \(m \equiv 0 \pmod 2 \) and \(s \equiv 0 \pmod 4\)}\label{smevensequiv0}

 We recall that \(RSpin(2c) = \mZ[\Lambda^1, \cdots, \Lambda^{c-2}, \Delta^{\pm}_c]\), where \(\Lambda^i\) is induced by the projection of \(Spin(2c)\) onto \(SO(2c)\) and \(\Delta^{\pm}_c\) arise from the Clifford algebra. Using similar argument as in the previous case, we have
\begin{theorem}\label{tRHmevens0mod4}
\[RH_{m,2s} = \mZ[\Lambda^1, \cdots, \Lambda^{c-2}, {\Delta}^{\pm}_c, \theta]/\langle \theta^2-1\rangle = \mZ[\bar{\Pi}_1, \cdots, \bar{\Pi}_{c-2}, \bar{\Delta}^{\pm}_c, y]/\langle y^2+2y\rangle,\]
with  
\[ 
\Delta_c^2 = \sum_{i=0}^c 2^{2(c-i)}\bar{\Pi}_i, \quad {\Delta}_c^+{\Delta}_c^- = \sum_{i=0}^{c-1}2^{2(c-1-i)}\bar{\Pi}_i  \quad \text{and} \quad \chi^2 = \bar{\Pi}_c. \]    
\end{theorem}
\begin{proposition}
The restriction map \(\Res: RSpin(m) = \mathbb{Z}[\Pi_1,\ldots, \Pi_{n-2}, \chi_n, \delta^{+}_n] \to RH_{m,2s} = \mZ[\bar{\Pi}_1, \cdots, \bar{\Pi}_{c-2}, \bar{\delta}^{\pm}_c, y]/\langle y^2+2y\rangle \) is determined in terms of its generators by
    \begin{align*}
    Res(\Pi_i) &= \sum_{j=0}^{c} \binom{s/2}{i-j}(-1)^{i-j-1}2^{2(i-j)-1}y \bar{\Pi}_j,\\
    Res(\chi_n) &= 0,\\
    Res(\delta_n^{+}) &= 2^{s-2}(y+2)\delta_c + 2^{n-2}y.
    \end{align*}
\end{proposition}
\begin{proof}
As in Proposition \ref{restriction map}, we have
\[Res(\Pi_i) = \sum_{j=0}^{c} \binom{s/2}{i-j}(-1)^{i-j-1}2^{2(i-j)-1}y \bar{\Pi}_j.\]
To calculate \(Res(\chi_n)\) and \(Res(\delta_n^{+})\), we consider the relation \(\Delta_n[t] = \prod_{j=1}^n(z_j+tz_j^{-1})\). \\
For \(t = \pm 1\), we have \(\Delta_n[t]=\Delta_n^{+}+t\Delta_n^-.\\\)
Hence for \(t = \pm 1,\)
\begin{align*}
    j^{\#}(\Delta_n[t]) &= j^{\#}  \left(\prod_{i=1}^n(z_i+tz_i^{-1}) \right)\\
    &= j^{\#}  \left(\prod_{i=1 \atop i \equiv 0 (2)}^s\left(z_i+tz_i^{-1}\right)\prod_{i=1 \atop i \equiv 1 (2)}^s\left(z_i+tz_i^{-1}\right)\prod_{i=s+1}^n\left(z_i+tz_i^{-1}\right) \right)\\
    &= j^{\#}  \left(\prod_{i=1 \atop i \equiv 0 (2)}^s\left(1+t\right)\prod_{i=1 \atop i \equiv 1 (2)}^s\left(-\phi+t(-\phi^3)\right)\prod_{i=s+1}^n\left(z_i+tz_i^{-1}\right) \right)\\
   &= (1+t)^{s/2}\phi^{s/2}\left(1 +t\theta\right)^{s/2}\Delta_c[t] \quad \quad \quad \quad \quad \quad \quad \quad \quad \quad \quad \quad \quad \quad (\text{From Lemma}~ \ref{limageofz_i})\\
   &=    \begin{cases}
   ~~ (1+t)^{s/2}\left(1 +t\theta\right)^{s/2}\Delta_c[t]  \quad \quad ~~~~~\text{for} ~ s/4 \equiv 0 \pmod 2\\
   ~~ (1+t)^{s/2}\theta\left(1 +t\theta\right)^{s/2}\Delta_c[t] \quad \quad ~~~~\text{for} ~ s/4 \equiv 1 \pmod 2
         \end{cases}\\
    &=    \begin{cases}
   ~~ 2^{s-2}(1+t)\left(1 +t\theta\right)\Delta_c[t]  \quad \quad ~~~~~\text{for} ~ s/4 \equiv 0 \pmod 2\\
   ~~ 2^{s-2}(1+t)\theta\left(1 +t\theta\right)\Delta_c[t] \quad \quad ~~~~\text{for} ~ s/4 \equiv 1 \pmod 2
         \end{cases} \quad \quad \quad (\text{From Lemma } \ref{simplifications})\\
    &=2^{s-2}\left((1+\theta)+t(1+\theta)\right)\Delta_c[t]\\
    &=2^{s-2}\left((1+\theta)+t(1+\theta)\right)\left(\Delta_c^{+}+t\Delta_{c}^{-1}\right)\\
    &= 2^{s-2}\left((1+\theta)\Delta_c + t(1+\theta)\Delta_c\right).
    \end{align*}
So, \[Res(\delta_n^{\pm}) = 2^{s-2}(y+2)\delta_c + 2^{n-2}y \quad \quad \text{and} \quad \quad Res(\chi_n)  = 0.\]
\\
Furthermore, we observe that the Lemma \ref{lrelationspibarandpi} works here as well and we have \[
\bar{\Pi}_i = \sum_{j =0}^{i} (-1)^j\binom{s/2 + j - 1}{j}(2y)^j\Pi^{'}_{i-j}.
\]
In fact, we have
\[
Tor_{RSpin(m)}^\ast(RH_{m, 2s}, \mZ) \equiv Tor^\ast_{A}(B, \mZ)
\]
where
\[
A= RSpin(m)/ \langle \Pi_1, \Pi_2, \ldots, \Pi_c\rangle \,\,\,\text{and} \,\,\,\, B= RH_{m,2s}/ \langle \Pi_1^\prime, \Pi_2^\prime, \ldots, \Pi_c^\prime \rangle.
\]

\begin{corollary}

In $B$, we have

\begin{enumerate}

\item For \( c+1 \leq i \leq n\), \[\displaystyle \quad \quad \quad \Pi^{'}_i = \binom{(s/2) + i -1}{i}2^{2i-1}y - \sum_{j=c+1}^{i-1} \binom{(s/2)+i-j-1}{i-j}2^{2(i-j)}\Pi_j^{'}.\]

\item \(\delta_c^2 + 2^{c+1}\delta_c = 2^{2c-1}y \left[1-\binom{(s/2)+c}{c}\right]\).\\

\item \(\delta_c^+\delta_c^{-} +2^{c-1}\delta_c = 2^{2c-3}\left[1-\binom{(s/2)+c-1}{c-1}\right]\).

\item \(\delta_c\delta_c^{+} = 2^{2c-3}y\left[1-\binom{(s/2)+c-1}{c-1}\right] - 2^{c-1}\delta_c + (\delta_c^{+})^2.\)

\item \(\chi^2 = -\binom{(s/2)+c-1}{c}2^{2c-1}y.\)

\end{enumerate}
\end{corollary}

\begin{proof}

Note that proof of \((1)\) follows the argument of Corollary \ref{generatorsofB} verbatim. The rest of the results are obtained by evaluating for \(\bar{\Pi}_{i}\) (determined in Lemma \ref{generatorsofB}) in the relations listed in Theorem \ref{tRHmevens0mod4}.

\end{proof}

To conclude,
\begin{proposition}

As a ring, \(B \cong \mZ[\delta_c^{+},\delta_c, y]/I,\) where \(I = \langle y^2+2y, \delta_c^2 + 2^{c+1}\delta_c-2^{2c-1}y\left[1-\binom{s/2+c}{c}\right], \delta_c^+\delta_c^{-} +2^{c-1}\delta_c - 2^{2c-3}\left[1-\binom{(s/2)+c-1}{c-1}\right], \delta_c\delta_c^{+} - 2^{2c-3}y\left[1-\binom{(s/2)+c-1}{c-1}\right] + 2^{c-1}\delta_c - (\delta_c^{+})^2 \rangle. \)

\end{proposition}
We now describe the structure of $A.$  Similar to the pervious case, we perform a series of transformations on the generators of A such that most of the new generators belong to the kernel of the restriction map, thereby, allowing us to lessen the computational complexity associated with \(Tor^\ast_{A}(B, \mZ).\) The justification of the following statements is same as in the case of \(K^{\ast}(FV_{m,2s})\) for \( m \equiv 1 \pmod 2 \) and \( s \equiv 0 \pmod 4.\)

\begin{enumerate}
\item First, we choose $P_{c+1}, \ldots, P_{n-2}$ in $A$ such that there is an isomorphism between $A= \mZ[\Pi_{c+1}, \ldots, \Pi_{n-2}, \chi_, \delta_n^{+}]$ and $\mZ[P_{c+1}, \ldots, P_{n-2}, \chi, \delta_n^{+}]$ and
\[
\Res(P_i)= \binom{(s/2) + i -1}{i}2^{2i-1}y.
\]
\item Then, we choose elements $U_{1}, U_{2}, \ldots, U_{s-2}$ such that there is an isomorphism $\Psi: A= \mZ[P_{c+1}, P_{c+2}, \ldots, P_{n-2}, \chi, \delta_n^{+}] \to \mZ[U_1, U_2, \ldots, U_{s-2}, \chi, \delta_n^{+}]$ with
\[
\quad \quad \quad \Res(U_1)= b_0y \,\,\,\, and \,\,\,\,\, \Res(U_l)= a_l b_0y\,\,\,\,\, \text{for all} \,\,\,\, l= 2, \ldots, s-2.
\]
where,
\begin{itemize}
\item $b_0$= GCD $\displaystyle{\bigg\{\binom{(s/2) + i -1}{i}2^{2i-1}\bigg|i= c+1, c+2, \ldots, n-2\bigg\}}.$
\item $a_l \in \mZ.$
\end{itemize}
\item Futher, we choose elements $V_1, V_2, \ldots, V_{s-2} \in A$ such that there is an isomorphism $\Theta \colon A= \mZ[U_1, U_2, \ldots, U_{s-2},\chi, \delta_n^{+}] \to \mZ[V_1, V_2, \ldots, V_{s-2}, \chi, \delta_n^{+}]$ with
\[
\Res(V_1)= b_0y, \,\,\,\,\, and \,\,\,\,\, \Res(V_l)= 0 \,\,\,\, \, \text{for all }\,\,\, 2 \leq l \leq s-2.
\]
\end{enumerate}

\begin{proposition}
    \[Tor^{\ast}_A(B;\mZ) \cong \Lambda^{*}_{\mZ}[t_1, \cdots, t_{s-2}] \otimes_{\mZ} Tor^{\ast}_{A'}(B; \mZ) \]
    with \(dim \, t_i = 1\) and \(A' = \mZ[V_1, \delta_n^{+}]\) 
\end{proposition}

\end{proof}
Now, we are required to compute the homology of the Koszul complex given by
\[ 
0 \xrightarrow{d_3} B(x_1\wedge x_2) \xrightarrow{d_2} Bx_1 \oplus Bx_2 \xrightarrow {d_1} B \rightarrow 0 
\]

Recall,    \(B \cong \mZ[\delta_c, y]/ I,\) where \\
$I = \bigg\langle y^2+2y, \delta_c^2 + 2^{c+1}\delta_c - 2^{2c-1}y \left[1-\binom{(s/2)+c}{c}\right], \delta_c\delta_c^{+} -2^{2c-3}y\left[1-\binom{(s/2)+c-1}{c-1}\right]+2^{c-1}\delta_c -(\delta_c^{+})^2 \bigg\rangle.$\\Here,

 \begin{align*}
d_1(x_1) & := Res(V_1)= b_0y, \\
d_1(x_2)&:= Res(\delta_n) =  2^{s-2}(2+y)\delta_c +2^{n-2}y.
\end{align*}
Therefore, 
\begin{align*}
H_0 &= \frac{B} {\langle 2^{\alpha}y,   2^{s-2}(2+y)\delta_c +2^{n-2}y\rangle B}~, \\
H_1 &= \frac{\langle u_1, u_2, u_3, u_4 \rangle B}{  (byx_2 - (2^{s-1}(2+y)\delta_c + 2^{n-1}y)x_1 )B}~,\\
H_2 &=  vB.  
\end{align*}  where \(2^{\alpha} = GCD(2^{n-1},b_0)\) and 
\begin{align*}
u_1 &= (y+2)x_1,\\
u_2 &= (y+2) (2^{c+1}+ \delta_c) x_2,\\
u_3 &= 2^{n- \alpha-1}x_1+ 2^{-\alpha}b_0yx_2,\\
u_4 &= 2^{n-\alpha-2}\delta_cx_1 - 2^{-\alpha}b_0(\delta_c +2^cy + 2^{c+1})x_2,\\
v &=(2+y)(2^{c+1}+\delta_c)x_1 \wedge x_2.
\end{align*}

\begin{theorem}
For \(m \equiv 0 \pmod 2 \) and \(s \equiv 0 \pmod 4\), let $2^{\alpha}= GCD\{2^{n-1}, b_0\}$ with $b_0= GCD \displaystyle{\bigg\{\binom{(s/2) + i -1}{i}2^{2i-1}\bigg|i= c+1, c+2, \ldots, n-2\bigg\}}.$ Then
\[
K^\ast(FV_{m,2s}) \cong \begin{cases}
\Lambda^{*}_{\mZ}[t_1, \cdots, t_{s-2}, u_1, u_2, u_3,v] \otimes_{\mZ} \mZ[\delta_c, y]/I,\\
\Lambda^{*}_{\mZ}[t_1, \cdots, t_{s-2}, u_1, u_2, u_3,u_4,v] \otimes_{\mZ} \mZ[\delta_c, y]/I,
\end{cases}
\]
where \(I\) is generated by the relations $y^2+2y, ~\delta_c^2+ 2^{c+1}\delta_c- 2^{2c-1}y \bigg[1- \binom{(s/2)+c}{c} \bigg], ~\delta_c\delta_c^{+} + 2^{c-1}\delta_c - 2^{2c-3}y\bigg[1-\binom{(s/2)+c-1}{c-1}\bigg]-(\delta_c^{+})^2, ~2^{\alpha}y, ~2^{s-2}(2+y)\delta_c+ 2^{n-2}y, ~2^{\alpha}u_3 - (2^{s-2} \delta_c + 2^{n-2})u_1, ~(y+2)u_3- 2^{n-\alpha-1}u_1, ~yu_1, ~yu_2, ~\delta_cu_2, ~(y+2)u_4- 2^{n-2-\alpha}\delta_cu_1- 2^{-\alpha}b_0u_2, ~(y+2)(2^c\delta_c+1)u_3 +(y+2)\delta_cu_4-2^{n-\alpha-1}u_1, ~2^{c-\alpha+1}b_0u_2 + (y+2)(2^{c+1} +\delta_c)u_4, ~2^{c-\alpha -1}b_0u_2 + (2^{c-1}\delta_c + 2^{2c-2}y\bigg[1- \binom{(s/2)+c}{c} \bigg] +2^{c-1}\delta_cy)u_3 +(2^c+\delta_c+2^cy+\delta_cy)u_4, ~u_2u_4, u_1u_3, \delta_cu_1u_4, ~2^{n-\alpha-2}b_0u_1u_2+2^nu_1u_4-b_0u_2u_3, ~2^{n-2\alpha-1}b_0u_1u_2+2^{-\alpha}b_0u_2u_3 +2u_3u_4, ~-2^{n-\alpha}u_1u_4 + 2^{-\alpha}b_0u_2u_3+2u_3u_4, ~2^{n-2\alpha-1}b_0u_1u_2 + 2^{n-\alpha-1}u_1u_4+2u_3u_4$ and $u_1u_2- 2v.$
\end{theorem}

\section{\(K^{\ast}(FV_{m,2s})\) for \(s \equiv 2 \pmod 4\)}\label{smoddsequiv2}
We know, from Remark \ref{romegasquare} that for $s \equiv 2 \pmod{4},$ we have $\omega^2= -1,$ which entails a significant difference in the characterisation of the ring \(RH_{m,2s}\) from that of \(s \equiv 0 \pmod 4.\)  We continue to use the commutative diagram \ref{dcommutative} , Property 1 and Property 2 in Section \ref{srep_rings} and we have the following analysis. Notice that since representation rings of $Spin(m)$ depend on parity of $m,$ we work out even and odd case separately. 

\subsection{For \(m \equiv 1 \pmod 2\)}
As in the previous two cases, we first determine \(RH_{m,2s}.\) Here \(Ker(\psi) = \{(1,1), (\tilde\omega^2,-1)\}\). 
We have similar descriptions for the rings \(R\tilde\Omega, RSpin(2c+1)\) and \(R(H_{m,2s}/Spin(2c+1))\) as in the case where \(s \equiv 0 \pmod 4.\) We recall that the representations \(\Lambda^i \in RSpin(2c+1)\) are induced by \(p: Spin(2c+1) \to SO(2c+1)\) whose kernel is \(\left\{\pm 1\right\}\). Thus, \(\Lambda^i(-1) = I\). Whereas, the representation \(\Delta_c\) comes from the operation of left multiplication of \(Spin(2c+1)\) on \(C^0_{2c+1}\), where \(C_{2c+1} (= C^0_{2c+1} \oplus C^1_{2c+1})\) is the real Clifford algebra, as discussed in Section 2, associated to the quadratic form \(Q(x)=-|x|^2\) with \(\mZ/2\mZ\)-grading. So \(\Delta_c(-1) = -I\).
We now pick generators of \(RH_{m,2s}\) from \(R\tilde\Omega \otimes RSpin(2c+1)\) by checking for triviality on  \(Ker(\psi)\) and  fixated representations of the codomain under \(r_h^{\ast}, \) for  (as illustrated in Section \ref{RHsubring}).\\

\[
\begin{tabular}{|c|c|c|c|c|c||c|c|c|c|c|}
 \hline
 $\rho$  & $\rho(\tilde\omega^2,-1)$ &  \multicolumn{4}{c|}{$\rho(g,g')$}& \multicolumn{4}{c|}{$\rho(\tilde\omega^2g,-g')$}& $r^{\ast}_{(\tilde\omega^2,-1)}$\\
 \hline
       &     & $g=1$ & $g=\tilde\omega$ & $g=\tilde\omega^2$ & $g=\tilde\omega^3$ & $g=1$ & $g=\tilde\omega$ & $g=\tilde\omega^2$ & $g=\tilde\omega^3$ & \\
 \hline
 $\Lambda^i$ & I  & $\Lambda^i$ & $\Lambda^i$ & $\Lambda^i$ & $\Lambda^i$ & $\Lambda^i$ & $\Lambda^i$ & $\Lambda^i$ & $\Lambda^i$ & $\Lambda^i$\\
 $\Delta_c$ & -I &  $\Delta_c$ &  $\Delta_c$ &  $\Delta_c$ &  $\Delta_c$ &  $-\Delta_c$ &  $-\Delta_c$ &  $-\Delta_c$ &  $-\Delta_c$ &  $-\Delta_c$\\
 $\phi \Lambda^i$ & -I & $\Lambda^i$ & $i\Lambda^i$ & $-\Lambda^i$ & $-i\Lambda^i$ & $-\Lambda^i$ & $-i\Lambda^i$ & $\Lambda^i$ & $i\Lambda^i$ & $-\phi\Lambda^i$\\
 $\phi \Delta^c$ & I &  $\Delta_c$ &  $i\Delta_c$ &  $-\Delta_c$ &  $-i\Delta_c$ &  $\Delta_c$ &  $i\Delta_c$ &  $-\Delta_c$ &  $-i\Delta_c$ &  $\phi\Delta_c$\\
 $\phi^2  \Lambda^i$ & I & $\Lambda^i$ & $-\Lambda^i$ & $\Lambda^i$ & $-\Lambda^i$ & $\Lambda^i$ & $-\Lambda^i$ & $\Lambda^i$ & -$\Lambda^i$ & $\phi^2\Lambda^i$\\
 $\phi^2  \Delta^c$ & -I &  $\Delta_c$ &  -$\Delta_c$ &  $\Delta_c$ &  $-\Delta_c$ &  $-\Delta_c$ &  $\Delta_c$ &  $-\Delta_c$&  -$\Delta_c$ & $-\phi^2 \Delta_c$\\
 $\phi^3  \Lambda^i$ & -I & $\Lambda^i$& $-i\Lambda^i$ & $-\Lambda^i$ & $i\Lambda^i$ & $-\Lambda^i$ & $i\Lambda^i$ & $\Lambda^i$ & $-i\Lambda^i$ & $-\phi^3\Lambda^i$\\
 $\phi^3 \Delta^c$ & I &  $\Delta_c$ &  $-i\Delta_c$ &  $-\Delta_c$ &  $i\Delta_c$ &  $\Delta_c$ &  $-i\Delta_c$ &  $-\Delta_c$ &  $i\Delta_c$ &  $\phi^3\Delta_c$
 \\
 \hline
\end{tabular}
\]
Following analogous notations as in Theorem \ref{tRHmodds0mod4},
\begin{theorem}\label{tRHmodds2mod4}
\[
RH_{m,2s} = \mZ[\Pi_1, \ldots, \Pi^{c-1}, \bar{\Delta}_c = \phi\Delta_c, y]/\langle y^2+2y\rangle, 
\]
with  \[ \bar{\Delta}_c^2 = \theta\sum_{i=0}^c 2^{2(c-i)}\bar{\Pi}_i.\]
    
\end{theorem}

Next, we obtain the  \(RSpin(m)\)-module structure of \(RH_{m,2s}\) via \(Res\) using the commutative diagram \ref{dcommutative} to be: \begin{proposition}\label{restriction map}
    Let \(m = 2n+1\) and \(s\equiv 2 \pmod4\). \(Res: RSpin(m) = \mathbb{Z}[\Pi_1,\ldots, \Pi_{n-1}, \delta_n] \to RH_{m,2s} = \mZ[\bar{\Pi}_1, \cdots, \bar{\Pi}_{c-1}, \bar{\Delta}_c, y+1]/ \langle y^2+2y \rangle\) is determined in terms of its generators by
    \begin{align*}
    Res(\Pi_i) &= \sum_{j=0}^{c} \binom{s/2}{i-j}(-1)^{i-j-1}2^{2(i-j)-1}y\bar{\Pi}_j,\\
    Res(\delta_n) &= -2^{s-1}(y+2)\delta_c - 2^{n-1}y.
    \end{align*}
\end{proposition}

By performing the series of transformations on the generators of \(RSpin(m)\) and \(RH_{m,2s}\) as in Section \ref{Torcomputation}, we obtain the following:
\begin{proposition}
    \[Tor^{\ast}_{RSpin(m)}(RH_{m,2s};\mZ) \cong \Lambda^{*}_{\mZ}[t_1, \cdots, t_{s-2}] \otimes_{\mZ} Tor^{\ast}_{A'}(B; \mZ) \]
    with \(dim \, t_i = 1\),    \(A' = \mZ[V_1, \delta_n]\) such that \(Res(V_1) = b_0y\) and
    \(B \cong \mZ[\delta_c, y]/\langle I \rangle \) where $I = \bigg\langle y^2+2y, ~\delta_c^2 + 2^{c+1}\delta_c - 2^{2c-1}y \left[1+\binom{(s/2)+c}{c}\right] \bigg\rangle.$
\end{proposition}
Consequently, we determine the homology of the Koszul complex given by
\[ 
0 \xrightarrow{d_3} B(x_1\wedge x_2) \xrightarrow{d_2} Bx_1 \oplus Bx_2 \xrightarrow {d_1} B \rightarrow 0 
\]
where,
 \begin{align*}
d_1(x_1) & := Res(V_1)= b_0y  ~~  ~ \text{where}~ b_0= GCD\left\{2^{2i-1}\binom{(s/2) + i -1}{i}\bigg | c+1 \leq i \leq n-1 \right\}. \\
d_1(x_2)&:= Res(\delta_n) =  -2^{s-1}(2+y)\delta_c -2^{n-1}y.
\end{align*}
Homology of this complex is computed to be:
\begin{align*}
H_0 &= \frac{B} {\langle 2^{\alpha}y,   2^{s-1}(2+y)\delta_c +2^{n-1}y \rangle B}~, \\
H_1 &= \frac{\langle u_1, u_2, u_3, u_4 \rangle B}{ (byx_2 + (2^{s-1}(2+y)\delta_c + 2^{n-1}y)x_1)B}~,\\
H_2 &=  vB.  
\end{align*}  where \(2^{\alpha} = GCD(2^n,b_0)\) and 
\begin{align*}
u_1 &= (y+2)x_1,\\
u_2 &= (y+2) (2^{c+1}+ \delta_c) x_2,\\
u_3 &= 2^{n- \alpha}x_1- 2^{-\alpha}b_0yx_2,\\
u_4 &= 2^{n-1-\alpha}\delta_cx_1 + 2^{-\alpha}b_0(\delta_c +2^cy + 2^{c+1})x_2, \\
v &= (2+y)(2^{c+1}+\delta_c)(x_1 \wedge x_2).
\end{align*}

\begin{theorem}
For \(m \equiv 1 \pmod 2 \) and \(s \equiv 1 \pmod 4\), let $2^{\alpha}= GCD\{2^n, b_0\}$ with $b_0= GCD\left\{2^{2i-1}\binom{(s/2) + i -1}{i}\bigg | c+1 \leq i \leq n-1 \right\}$  we have
\[
K^\ast(FV_{m,2s}) \cong \begin{cases}
\Lambda^{*}_{\mZ}[t_1, \cdots, t_{s-2}, u_1, u_2, u_3,v] \otimes_{\mZ} \mZ[\delta_c, y]/I,\\
\Lambda^{*}_{\mZ}[t_1, \cdots, t_{s-2}, u_1, u_2, u_3,u_4,v] \otimes_{\mZ} \mZ[\delta_c, y]/I,
\end{cases}
\]
where \(I\) is generated by the relations
$y^2+2y, ~\delta_c^2+ 2^{c+1}\delta_c- 2^{2c-1}y \bigg[1+ \binom{(s/2)+c}{c} \bigg], ~2^{\alpha}y, ~2^{s-1}(2+y)\delta_c+ 2^{n-1}y, ~-2^{\alpha}u_3 + (2^{s-1} \delta_c + 2^{n-1})u_1, ~(y+2)u_3- 2^{n-\alpha}u_1, ~yu_1, ~yu_2, ~\delta_cu_2, ~(y+2)u_4- 2^{n-1-\alpha}\delta_cu_1 -2^{-\alpha}b_0u_2, ~(y+2)(2^c\delta_c+1)u_3 +(y+2)\delta_cu_4-2^{n-\alpha}u_1, ~2^{c-\alpha+1}b_0u_2 - (y+2)(2^{c+1} +\delta_c)u_4, ~-2^{c-\alpha -1}b_0u_2 + (2^{c-1}\delta_c + 2^{2c-2}y\bigg[1+ \binom{(s/2)+c}{c} \bigg] +2^{c-1}\delta_cy)u_3 + (2^c+\delta_c+2^cy+\delta_cy)u_4, ~u_2u_4, ~u_1u_3, ~\delta_cu_1u_4, ~2^{n-\alpha}b_0u_1u_2-2^nu_1u_4+b_0u_2u_3, ~2^{n-2\alpha}b_0u_1u_2+2^{-\alpha}b_0u_2u_3 -2u_3u_4, ~-2^{n-\alpha+1}u_1u_4 - 2^{-\alpha}b_0u_2u_3+2u_3u_4, ~2^{n-2\alpha}b_0u_1u_2 - 2^{n-\alpha}u_1u_4-2u_3u_4,$ and $u_1u_2- 2v.$
\end{theorem}

\subsection{For \(m \equiv 0 \pmod 2 \)}\label{smevensequiv2}
In this case, mimicking calculations done earlier, we get
\begin{theorem}\label{tRHmevens2mod4}
\[RH_{m,2s} = \mZ[\Lambda^1, \cdots, \Lambda^{c-2},  \bar{\Delta}^{\pm}_c=\phi\Delta^{\pm}_c, \theta]/\langle\theta^2-1 \rangle = \mZ[\bar{\Pi}_1, \cdots, \bar{\Pi}_{c-2}, \bar{\Delta}^{\pm}_c, y]/\langle y^2+2y \rangle,\]
with  \[ \bar{\Delta}_c^2 = \theta \sum_{i=0}^c 2^{2(c-i)}\bar{\Pi}_i, \quad \bar{\Delta}_c^+\bar{\Delta}_c^- = \theta\sum_{i=0}^{c-1}2^{2(c-1-i)}\bar{\Pi}_i  \quad \text{and} \quad \chi^2 = \theta\bar{\Pi}_c. \]    
\end{theorem}
The restriction map is given by,
\begin{proposition}
The restriction map \(Res: RSpin(m) = \mathbb{Z}[\Pi_1,\ldots, \Pi_{n-2}, \chi_n, \delta^{+}_n] \to RH_{m,2s} = \mZ[\bar{\Pi}_1, \cdots, \bar{\Pi}_{c-2}, \bar{\delta}^{\pm}_c, y]/(y^2+2y)\) is determined in terms of its generators by
    \begin{align*}
    Res(\Pi_i) &= \sum_{j=0}^{c} \binom{s/2}{i-j}(-1)^{i-j-1}2^{2(i-j)-1}y \bar{\Pi}_j,\\
    Res(\chi_n) &= 0,\\
    Res(\delta_n^{+}) &= -2^{s-2}(y+2)\delta_c - 2^{n-2}y.
    \end{align*}
\end{proposition}
Adapting the same procedure as  in previous cases, we leverage  the module structure of \(RH_{m,2s}\) and reduce the computation of \(Tor^{\ast}_{RSpin(m)}(RH_{m,2s};\mZ)\) as follows:
\begin{proposition}
    \[Tor^{\ast}_{RSpin(m)}(RH_{m,2s};\mZ) \cong \Lambda^{*}_{\mZ}[t_1, \cdots, t_{s-2}] \otimes_{\mZ} Tor^{\ast}_{A'}(B; \mZ) \]
    with \(dim \, t_i = 1\),    \(A' = \mZ[V_1, \delta_n]\) such that \(Res(V_1) = b_0y\) and
    \(B \cong \mZ[\delta_c, y]/ I  \) where $b_0= GCD\left\{2^{2i-1}\binom{(s/2) + i -1}{i}\bigg | c+1 \leq i \leq n-1 \right\}$ and $I = \bigg\langle y^2+2y, \delta_c^2 + 2^{c+1}\delta_c - 2^{2c-1}y \left[1+\binom{(s/2)+c}{c}\right], \delta_c\delta_c^{+} +2^{c-1}\delta_c - 2^{2c-3}\left[1+\binom{(s/2)+c-1}{c-1}\right] -(\delta_c^{+})^2 \bigg\rangle.$
\end{proposition}

This gives the Koszul complex 
\[ 
0 \xrightarrow{d_3} B(x_1\wedge x_2) \xrightarrow{d_2} Bx_1 \oplus Bx_2 \xrightarrow {d_1} B \rightarrow 0 
\]

where,

 \begin{align*}
d_1(x_1) & := Res(V_1)= b_0y, \\
d_1(x_2)&:= Res(\delta_n) =  -2^{s-2}(2+y)\delta_c -2^{n-2}y. \\ 
\end{align*}

We determined its homology groups as:
\begin{align*}
H_0 &= \frac{B} {\langle 2^{\alpha}y,   2^{s-2}(2+y)\delta_c +2^{n-2}y \rangle B}~, \\
H_1 &= \frac{\langle u_1, u_2, u_3, u_4 \rangle B}{(byx_2 + (2^{s-2}(2+y)\delta_c + 2^{n-2}y)x_1) B}~,\\
H_2 &= vB.  
\end{align*}  where \(2^{\alpha} = GCD(2^{n-1},b_0)\) and 
\begin{align*}
u_1 &= (y+2)x_1,\\
u_2 &= (y+2) (2^{c+1}+ \delta_c) x_2,\\
u_3 &= 2^{n- \alpha -1}x_1- 2^{-\alpha}b_0yx_2,\\
u_4 &= 2^{n-\alpha-2}\delta_cx_1 + 2^{-\alpha}b_0(\delta_c +2^cy + 2^{c+1})x_2,\\
v &= (2+y)(2^{c+1}+ \delta_c)( x_1 \wedge x_2).
\end{align*} 

\begin{theorem}
For \(m \equiv 1 \pmod 2 \) and \(s \equiv 1 \pmod 4\), let $2^{\alpha}= GCD\{2^{n-1}, b_0\}$ with $b_0= GCD\left\{2^{2i-1}\binom{(s/2) + i -1}{i}\bigg | c+1 \leq i \leq n-2 \right\}.$ Then
\[
K^\ast(FV_{m,2s}) \cong \begin{cases}
\Lambda^{*}_{\mZ}[t_1, \cdots, t_{s-2}, u_1, u_2, u_3,v] \otimes_{\mZ} \mZ[\delta_c, y]/I,\\
\Lambda^{*}_{\mZ}[t_1, \cdots, t_{s-2}, u_1, u_2, u_3,u_4,v] \otimes_{\mZ} \mZ[\delta_c, y]/I,
\end{cases}
\]
where \(I\) is generated by
$y^2+2y, ~\delta_c^2+ 2^{c+1}\delta_c- 2^{2c-1}y \bigg[1+ \binom{(s/2)+c}{c} \bigg], ~\delta_c\delta_c^{+} + 2^{c-1}\delta_c - 2^{2c-3}y\bigg[1+\binom{(s/2)+c-1}{c-1}\bigg]-(\delta_c^{+})^2, ~2^{\alpha}y, ~2^{s-2}(2+y)\delta_c+ 2^{n-2}y, ~-2^{\alpha}u_3 + (2^{s-2} \delta_c + 2^{n-2})u_1, ~(y+2)u_3- 2^{n-\alpha-1}u_1, ~yu_1, ~yu_2, ~\delta_cu_2, ~(y+2)u_4 - 2^{n-2-\alpha}\delta_cu_1- 2^{-\alpha}b_0u_2, ~(y+2)(2^c\delta_c+1)u_3 -(y+2)\delta_cu_4-2^{n-\alpha-1}u_1, ~2^{c-\alpha+1}b_0u_2 - (y+2)(2^{c+1} +\delta_c)u_4, ~-2^{c-\alpha -1}b_0u_2 + (2^{c-1}\delta_c + 2^{2c-2}y\bigg[1+ \binom{(s/2)+c}{c} \bigg] +2^{c-1}\delta_cy)u_3 + (2^c+\delta_c+2^cy+\delta_cy)u_4, ~u_2u_4, ~u_1u_3, ~\delta_cu_1u_4, ~2^{n-\alpha-2}b_0u_1u_2-2^nu_1u_4-b_0u_2u_3, ~2^{n-2\alpha-1}b_0u_1u_2+2^{-\alpha}b_0u_2u_3 -2u_3u_4, ~-2^{n-\alpha}u_1u_4 - 2^{-\alpha}b_0u_2u_3+2u_3u_4, ~2^{n-2\alpha-1}b_0u_1u_2 - 2^{n-\alpha-1}u_1u_4-2u_3u_4,$ and $u_1u_2- 2v.$
\end{theorem}

We conclude our discussion with an immediate consequence which is the following corollary which describes order of the complexification of canonical line bundle and the tangent bundle in the reduced $K$ ring, denoted by $\bar{K}^\ast (FV_{m,2s}).$

\begin{corollary}
(1) We have a canonical (real) line bundle $\xi_{m,2s}$ associated to the double cover $V_{m,2s} \to FV_{m,2s}$ (Ref \cite[Lemma 11]{bgs}). It follows that the element $y+1 \in K^\ast(FV_{m,2s})$ can be identified with the complexification $c(\xi_{m,2s})$ of $\xi_{m,2s}).$\\
(2) From \cite[Lemma 11, 12]{bgs}, we know that the tangent bundle $TFV_{m,2k}$ satisfies a relation given by
\begin{equation}\label{stableiso}
\frac{s(s+1)}{2}\big((\epsilon_{\mR} \oplus \xi_{m,2s}) \otimes  (\epsilon_{\mR} \oplus \xi_{m,2s})\big)  \oplus TFV_{m,2s} \cong ms(\epsilon_{\mR} \oplus \xi_{m,2s}) \oplus s\epsilon_{\mR}.
\end{equation}
So, in $KO^\ast(FV_{m,2s})$ we have
\[
s(s+1)[(\epsilon_{\mR} \oplus \xi_{m,2s})]+ [TFV_{m,2s}]= ms[(\epsilon_{\mR} \oplus \xi_{m,2s})] + [s\epsilon_{\mR}].
\]
Further simplification gives us
\[
[TFV_{m,2s}]= s(m-s)+ s(m-s-1)[\xi_{m,2s}]
\]
So,  in $\bar{K}(FV_{m,2s)},$ we have
\[
[c(TFV_{m,2s})]= s(m-s-1)[y].
\]
\end{corollary}

\section*{Appendix A}\label{appendix}
In general, identifying generators of  polynomial ring in several variables over a field or a ring  by inspection could be a tedious task. Here, we share a systematic way for determining generators using the notion of Grobner basis. 

Classically, Grobner basis helps in determining the generating set of an ideal in a polynomial ring \(\mF[x_1, \ldots, x_n]\) over a field \(\mF\) \cite[Section 9.6]{df}.  Additionally, in \cite{al}, an in-depth exploration of extending the notion of Grobner bases from fields to rings is provided. In this appendix, our discussions on Grobner basis in within the context of Lemma \ref{lkergenerators}, that is, for the broader set up of polynomial ring in several variables over a Noetherian commutative ring. To elucidate our approach to utilizing Grobner bases for generating sets, we introduce several key terminologies.
 \begin{itemize}
\item Any polynomial \(f \in B\) is a finite sum of terms of the form  \(a\delta_c^{r_1}y^{r_2}\), where \(r_i \in \{0,1\}\). 
\item We define degree \textit{lexicographical order} (\cite{al} pg: 19) on the ring \(B\) with \( \delta_c > y\). So, we have \[ 1 < y < \delta_c < \delta_cy.\]
\item We define {\it position over term ordering} (\cite{al} pg: 142) on the module \(Bx_1 \oplus Bx_2\) with \(x_1 > x_2\). Hence we have, \[x_2 <yx_2 < \delta_cx_2 < \delta_cyx_2 < x_1 < yx_1 < \delta_cx_1 < \delta_cyx_1.\]
\item Let \(lt(f)\) denotes the leading term of the polynomial \(f \in Bx_1 \oplus Bx_2\) subject to the above defined position over term order. 
For \(W \subseteq Bx_1 \oplus Bx_2. \)  Let  \(Lt(W) = \langle \{ lt(w)| w \in W\} \rangle\) denotes the ideal generated by the leading terms of the polynomials in \(W\).
\end{itemize} 
Now, considering \(Bx_1 \oplus Bx_2 \) as the direct sum of rings, \(B*8\), we apply the following necessary and sufficient condition on Grobner basis of an ideal to deduce a minimal generating set of \(Ker(d_1)\). 
\begin{theorem}\cite[Page 208]{al} \label{tgrobner}
Let \(R\) be a ring. Let I be an ideal of \(A = R[x_1, x_2, \ldots, x_n]\). Let \(G=\{g_1, g_2, \ldots, g_t\}\) be a set of non-zero polynomials in \(I\). Then G is a grobner basis of \(I\) if and only if \(Lt(G) = Lt(I).\)
\end{theorem}

Utilizing the sufficient condition outlined in Theorem \ref{tgrobner} regarding the leading terms of \(I\), we proceed to choose a potential set of generators for \(Ker(d_1)\). We explain the procedure step by step below, in the context of Lemma \ref{lkergenerators}.

\begin{steps}
\item Initially, since the monomial \(\delta_cyx_1\) is assigned the highest order, our aim is to identify a polynomial within the kernel that satisfies the requisite relations while possessing a minimal non-zero coefficient for the term \(\delta_cyx_1\).  Thus, we propose a candidate  \(g_1:= (y+2)\delta_cyx_1.\)\\
\item Subsequently, we aim to select a polynomial corresponding to the monomial \(\delta_cx_1\).Here, our objective is to find a polynomial from the kernel with a leading term of \(\delta_cx_1\) and a minimal coefficient. We have several options available, among which we choose one, such as \(g_2:=2^{n-\alpha} \delta_cx_1+ 2^{-\alpha}b_0\delta_cyx_2\) for \(\alpha = n\) and for \(\alpha < n\), we pick \(g_3:=2^{n-1-\alpha}\delta_cx_1 + 2^{-\alpha}b_0(\delta_c +2^cy + 2^{c+1})x_2\), where \(2^{\alpha} = gcd\{2^n, b_0\}\).\\
\item Next, we pick  a polynomial corresponding to \(yx_1\). Here  \(g_4:=(y+2)x_1\) serves as a suitable choice. It is worth noting that \(lt(g_3)\) makes \(g_1\) a redundant choice.\\
\item Continuing, we proceed to pick a polynomial corresponding to the monomial \(x_1\), resulting in \(g_5:=2^{n - \alpha}x_1 + 2^{-\alpha}b_0yx_2\). Again, we observe that \(lt(g_4)\) makes \(g_2\) a redundant choice .\\
\item Moving forward, we pick a polynomial with leading term as \(\delta_cyx_2\) and we have a minimal choice \(g_6:=\delta_cyx_2 + 2\delta_cx_2 +2^{c+1}yx_2 + 2^{c+2}x_2\).\\
\item Finally, we observe that there are no other polynomials in \(Ker(d_1)\) with leading term as \(\delta_cx_2, yx_2\) or \( x_2\).
\end{steps}

Thus, as in Lemma \ref{lkergenerators},\\
For \(\alpha = n\),  we have the generating set,
\begin{align*}
u_1 :=&  (y+2)x_1  &d_1(u_1) = 0\\
u_2:= &(y+2)(2^{c+1}+ \delta_c)x_2 &d_1(u_2) = 0\\
u_3:= & 2^{n-\alpha}x_1 +2^{-\alpha} b_0yx_2  &d_1(u_3) = 0
\end{align*} 
For \(\alpha < n\),  we have the generating set,
\begin{align*}
u_1 :=&  (y+2)x_1  &d_1(u_1) = 0\\
u_2:= &(y+2)(2^{c+1}+ \delta_c)x_2 &d_1(u_2) = 0\\
u_3:= & 2^{n-\alpha}x_1 +2^{-\alpha} b_0yx_2  &d_1(u_3) = 0\\
u_4:=& 2^{n-1-\alpha}\delta_cx_1 - 2^{-\alpha}b_0(\delta_c +2^cy + 2^{c+1})x_2 &d_1(u_4) = 0
\end{align*}
The table below outlines the various scenarios for obtaining polynomials that meet the condition of having minimal non-zero coefficients on specific leading terms while belonging to \(Ker(d_1)\).

\begin{landscape}
\begin{table}
\begin{tabular}{|c|c|c|c|c|c|c|c|>{\footnotesize}p{7.0cm}|}

\hline
$ p_1$ &$ p_2$ & $p_3$ & $q_1$ & $q_3$ &$ p_4$ & $q_2$ & $q_4$& Element of $Ker(d_1)$ \\
\hline
\hline
 $ 2^{n-\alpha} $ & $0$ & $0$ & $0$ &  $0$  & $ -2^{n+c-\alpha+2}$ & $ 2^{-\alpha+2}b_0$ & $2^{c+3-\alpha}b_0$& $2^{n-\alpha}(\delta_cy -2^{c+2})x_1 + 2^{-\alpha+2}b_0(\delta_c +2^{c+1})x_2$ \\
\hline
 $ 0 $ & $ 2^{n-\alpha}$ & 0 & 0 &  0  & $ 2^{n+c-\alpha+1}$ & $-2^{-\alpha+1}b_0$ & $-2^{c+2-\alpha}b_0$& $2^{n-\alpha}(\delta_c + 2^{c+1})x_1  - 2^{-\alpha}b_0(2\delta_c +2^{c+2})x_2 $ \\ 
 \hline
 $ 0 $ & $ 2^{n-1-\alpha}$ & 0 & 0 &  0  & $ 2^{n+c-\alpha}$ & $-2^{-\alpha}b_0$ & $-2^{c+1-\alpha}b_0$& $2^{n-\alpha}(2^{-1}\delta_c + 2^{c})x_1  - 2^{-\alpha}b_0(\delta_c +2^{c+1})x_2 $ \\ 
\hline
 $ 0 $ &$ 0$& $1$ & 0 &  0  & $ 2 $ & $0$ & $0$& $(y+2)x_1$ \\
\hline
 $ 0 $ &$ 0$& $0$ & $2^{-\alpha}b_0$ &  $0$  & $ -2^{n+c-\alpha+1}$ & $2^{-(\alpha-1)}b_0$ & $2^{c+2-\alpha}b_0$ & $-2^{n+c-\alpha+1}x_1 + 2^{-\alpha}b_0(\delta_cy -2^{n+c+1}\delta_c + 2^{c+2})x_2$ \\
\hline
 $ 0 $ &$ 0$& $0$ & $0$ &  $2^{-\alpha}b_0$  & $ 2^{n-\alpha}$ & $0$ & $0$&  $2^{n-\alpha}x_1 + 2^{-\alpha}b_0yx_2$\\
\hline
\hline
 $ 0 $ &$ 0$& $0$ & $1$ &  $2^{c+1}$  & $ 0 $ & $2$ & $2^{c+2}$& $ (\delta_cy +2\delta_c +2^{c+1}y +2^{c+2})x_2$\\
\hline
 $ 0 $ &$ 0$& $1$ & $0$ &  $2^{-\alpha}b_0$  & $ 2^{n-\alpha}+2$ & $0$ & $0$& $(y+2^{n-\alpha}+2)x_1+ 2^{-\alpha}b_0yx_2$ \\
\hline
 $ 0 $ &$ 0$& $2^{n+c-\alpha-1}$ & $2^{-\alpha}b_0$ &  $0$  & $ -2^{n+c-\alpha}$ & $2^{-\alpha+1}b_0$ & $2^{c+2-\alpha}b_0$& $ (2^{n+c-\alpha}(2^{-1}y-1)x_1 + 2^{-\alpha}b_0(\delta_cy +2\delta_c + 2^{c+2})x_2 $ \\
\hline
$ 0 $ &$ 2^{n-\alpha}$& $0$ & $0$ &  $-2^{c-\alpha+1}b_0$  & $ 0$ & $-2^{-\alpha+1}b_0$ & $-2^{c+2-\alpha}b_0$& $ 2^{n-\alpha}\delta_c x_1 - 2^{-\alpha}b_0(2\delta_c +2^{c+1}y +2^{c+2})x_2$ \\
\hline
$ 0 $ &$ 2^{n-1-\alpha}$& $0$ & $0$ &  $-2^{c-\alpha}b_0$  & $ 0$ & $-2^{-\alpha}b_0$ & $-2^{c+1-\alpha}b_0$& $ 2^{n-1-\alpha}(\delta_c )x_1 - 2^{-\alpha}b_0(\delta_c +2^{c}y +2^{c+1})x_2$ \\
\hline
$ 0 $ &$2^{n-\alpha}$& $0$ & $2^{-\alpha}b_0$ &  $0$  & $ 0$ & $0$ & $0$& $2^{n-\alpha} \delta_cx_1+ 2^{-\alpha}b_0\delta_cyx_2$ \\
\hline
$ 0 $ &$2^{n-1-\alpha}$& $0$ & $2^{-\alpha}b_0$ &  $0$  & $ -2^{n+c-\alpha}$ & $2^{-\alpha}b_0$ & $2^{c+1-\alpha}b_0$& $2^{n-1-\alpha} (\delta_c+ 2^{c+1})x_1+ 2^{-\alpha}b_0(\delta_cy +\delta_c +2^{c+1})x_2$ \\
\hline
$ 0 $ &$ 2^{n-\alpha}$& $2^{n+c-\alpha-1}$ & $0$ &  $0$  & $ 3.2^{n+c-\alpha}$ & $-2^{-\alpha+1}b_0$ & $-2^{c+2-\alpha}b_0$& $2^{n-\alpha}(\delta_c+2^{c-1}y +3.2^{c})x_1 -2^{-\alpha}b_0(2\delta_c+2^{c+2})x_2 $\\
\hline
$ 0 $ &$ 2^{n-1-\alpha}$& $-2^{n+c-\alpha-1}$ & $0$ &  $0$  & $ 0$ & $-2^{-\alpha}b_0$ & $-2^{c+1-\alpha}b_0$& $2^{n-\alpha}(2^{-1}\delta_c-2^{c-1}y)x_1 -2^{-\alpha}b_0(\delta_c+2^{c+1})x_2 $\\
\hline
$ 2^{n-\alpha} $ &$ 0$& $0$ & $0$ &  $2^{c-\alpha}b_0$  & $ -3(2^{n+c-\alpha})$ & $2^{-\alpha+2}b_0$ & $2^{c-\alpha+3}b_0$ & $ 2^{n-\alpha}(\delta_cy - 3(2^c))x_1 + 2^{-\alpha}b_0(2^2\delta_c+2^cy+2^{c+3})x_2$ \\
\hline
$ 2^{n-\alpha} $ &$ 0$& $0$ & $2^{-\alpha}b_0$ &  $0$  & $ -3(2^{n+c-\alpha+1})$ & $3.2^{-\alpha+1}b_0$ & $3.2^{-\alpha+c+2}b_0$ & $2^{n-\alpha}(\delta_cy -3(2^{c+1}))x_1 +2^{-\alpha}b_0(\delta_cy + 6\delta_c +3.2^{c+2})x_2$ \\
\hline
$ 2^{n-\alpha} $ &$ 0$& $2^{n+c-\alpha-1}$ & $0$ &  $0$  & $ -3(2^{n+c-\alpha})$ & $2^{-\alpha+2}b_0$ & $2^{c-\alpha+3}b_0$& $2^{n-\alpha}(\delta_cy + 2^{c-1}y -3(2^{c}))x_1 + 2^{-\alpha}b_0(2^2\delta_c + 2^{c+3})x_2$\\
\hline
$ 1 $ &$ 2 $& $0$ & $0$ &  $0$  & $ 0$ & $0$ & $0$& $(\delta_cy+2\delta_c)x_1$ \\
\hline
\hline
$ 0 $ &$ 2^{n-\alpha} $& $  2^{n+c-\alpha-1}$ & $0$ &  $ 2^{c-\alpha}b_0$  & $  2^{n+c-\alpha+2}$ & $-2^{-\alpha+1}b_0$ & $-2^{c+2-\alpha}b_0$& $ 2^{n-\alpha}(\delta_c +2^{c-1}y + 2^{c+2})x_1 - 2^{-\alpha}b_0(2\delta_c - 2^cy +2^{c+2})x_2$ \\
\hline
$ 0 $ &$ 2^{n-1-\alpha} $& $  -2^{n+c-\alpha-1}$ & $0$ &  $ 2^{c-\alpha}b_0$  & $  2^{n+c-\alpha}$ & $-2^{-\alpha}b_0$ & $-2^{c+1-\alpha}b_0$& $ 2^{n-\alpha}(2^{-1}\delta_c -2^{c-1}y + 2^{c})x_1 - 2^{-\alpha}b_0(\delta_c - 2^cy +2^{c+1})x_2$ \\
\hline
$ 2^{n-\alpha}  $ &$ 0 $& $ 2^{n+c-\alpha-1}$ & $0$ &  $2^{c-\alpha}b_0$  & $ -2^{n+c-\alpha+1}$ & $2^{-\alpha+2}b_0$ & $2^{c-\alpha+3}b_0$& $ 2^{n-\alpha}(\delta_cy +2^{c-1}y -2^{c+1})x_1 + 2^{-\alpha}b_0(2^2\delta_c + 2^cy + 2^{c+3})x_2$\\
\hline
$ 1 $ &$ 2 $& $ 2^{c+1}$ & $0$ &  $0$  & $ 2^{c+2}$ & $0$ & $0$&  $ (\delta_cy +2\delta_c +2^{c+1}y+2^{c+2})x_1$\\
\hline
$ 0 $ &$ 2^{n-\alpha} $& $ 2^{n+c-\alpha-1}$ & $ 2^{-\alpha}b_0$ &  $0$ & $2^{n+c-\alpha}$  & $0$ & $0$& $ 2^{n-\alpha}(\delta_c + 2^{c-1}y + 2^c)x_1 + 2^{-\alpha}b_0\delta_cyx_2$ \\
\hline
$ 0 $ &$ 2^{n-1-\alpha} $& $ 2^{n+c-\alpha-1}$ & $ 2^{-\alpha}b_0$ &  $0$ & $0$  & $2^{-\alpha}b_0$ & $2^{c+1-\alpha}b_0$& $ 2^{n-\alpha}(2^{-1}\delta_c + 2^{c-1}y)x_1 + 2^{-\alpha}b_0(\delta_cy +\delta_c +2^{c+1})x_2$ \\
\hline
$ 0 $ &$ 0 $& $ 2^{n+c-\alpha-1}$ & $ 2^{-\alpha}b_0$ &  $2^{c-\alpha}b_0$  & $ 0$ & $2^{-\alpha+1}b_0$ & $2^{c-\alpha+2}b_0$& $ 2^{n+c-\alpha-1}yx_1+2^{-\alpha}b_0(\delta_cy +2\delta_c +2^cy+2^{c+2})x_2 $ \\
\hline

\hline
\end{tabular}
\caption{A possible solution set}
\end{table}
\end{landscape}
\begin{landscape}
\begin{table}
  \centering

\begin{tabular}{|c|c|c|c|c|c|c|c|>{\footnotesize}p {8.5 cm}|}
\hline
$ p_1$ &$ p_2$ & $p_3$ & $q_1$ & $q_3$ &$ p_4$ & $q_2$ & $q_4$& Element of $Ker(d_1)$ \\
\hline
$ 2^{n-\alpha} $ &$ 0 $& $ 0 $ & $ 2^{-\alpha}b_0$ &  $ 2^{c-\alpha}b_0 $  & $ -5.2^{n+c-\alpha}$ & $3.2^{-\alpha+1}b_0$ & $3.2^{-\alpha + c +2}b_0$& $2^{n-\alpha}(\delta_cy -5.2^{c})x_1+2^{-\alpha}b_0(\delta_cy +6\delta_c +2^cy + 3.2^ {c+2} )x_2$\\
\hline
$ 2^{n-\alpha} $ &$ 2^{n-\alpha} $& $ 0 $ & $ 0$ &  $ 2^{c-\alpha}b_0 $  & $ -2^{n+c-\alpha}$ & $2^{-\alpha+1}b_0$ & $2^{c-\alpha+2}b_0$& $ 2^{n-\alpha}(\delta_cy +\delta_c -2^c)x_1 + 2^{-\alpha}b_0(2\delta_c + 2^cy + 2^{c+2})x_2$ \\
\hline
$ 2^{n-\alpha} $ &$ 0 $& $ 2^{n+c-\alpha-1} $ & $ 2^{-\alpha}b_0$ &  $ 0 $  & $-5.2^{n+c-\alpha}$ & $-3.2^{-\alpha+1}b_0$ & $-3.2^{c-\alpha+2}b_0$& $ 2^{n-\alpha}(\delta_cy+2^{c-1}y-5.2^{c})x_1+2^{-\alpha}b_0(\delta_cy - 6\delta_c -3.2^{c+2})x_2$ \\
\hline
$ 2^{n-\alpha} $ &$ 2^{n-\alpha} $& $ 0 $ & $2^{-\alpha}b_0$ &  $ 0 $  & $ -2^{n+c-\alpha+2}$ & $2^{-\alpha+2}b_0$ & $2^{c-\alpha+3}b_0$& $ 2^{n-\alpha}(\delta_cy +\delta_c - 2^{c+2})x_1+2^{-\alpha}b_0(\delta_cy +2^2\delta_c + 2^{c+3})x_2$\\
\hline
$ 0 $ &$ 2^{n-\alpha} $& $ 0 $ & $ 2^{-\alpha}b_0$ &  $ 2^{c-\alpha}b_0 $  & $ 2^{n+c-\alpha}$ & $0$ & $0$&  $ 2^{n-\alpha}(\delta_c +2^c)x_1+2^{-\alpha}b_0(\delta_cy+2^cy)x_2$\\
\hline
$ 0 $ &$ 2^{n-1-\alpha} $& $ 0 $ & $ 2^{-\alpha}b_0$ &  $ 2^{c-\alpha}b_0 $  & $0$ & $2^{-\alpha}b_0$ & $2^{c+1-\alpha}b_0$&  $ 2^{n-\alpha}(2^{-1}\delta_cx_1+2^{-\alpha}b_0(\delta_cy+\delta_c+2^cy +2^{c+1})x_2$\\

\hline

$ 2^{n-\alpha}$ &$2^{n-\alpha}$ & $2^{n+c-\alpha-1}$ & $2^{-\alpha}b_0$ & $0$ &$-3.2^{n+c-\alpha}$ & $2^{-\alpha+2}b_0$ & $2^{c-\alpha+3}b_0$& $2^{n-\alpha}(\delta_cy + \delta_c +2^{c-1}y - 3.2^c)x_1+2^{-\alpha}b_0(\delta_cy + 4\delta_c +2^{c+3})x_2$ \\
\hline 
$ 2^{n-\alpha}$ &$ 2^{n-\alpha}$ & $2^{n+c-\alpha-1}$ & $0$ & $2^{-\alpha}b_0$ &$ 0 $ & $ 2^{-\alpha+1}b_0$ & $2^{c-\alpha+2}b_0$& $2^{n-\alpha}(\delta_cy + \delta_c +2^{c-1}y)x_1+2^{-\alpha}b_0(2\delta_c +2^cy+ 2^{c+2})x_2$ \\
\hline
$ 2^{n-\alpha}$ &$ 2^{n-\alpha}$ & $0$ & $2^{-\alpha}b_0$ &$ 2^{c-\alpha}b_0$ & $-3.2^{n+c-\alpha}$ &  $ 2^{-\alpha+2}b_0$ & $2^{c-\alpha+2}b_0$& $2^{n-\alpha}(\delta_cy + \delta_c -3.2^{c-\alpha})x_1+2^{-\alpha}b_0(\delta_cy + 2^2\delta_c +2^cy+ 2^{c+3})x_2$ \\
\hline
$ 2^{n-\alpha}$ &$0$ & $2^{n+c-\alpha-1}$ & $2^{-\alpha}b_0$ &$ 2^{c-\alpha}b_0$ & $ 2^{n+c-\alpha+2}$ & $ 3.2^{-\alpha+1}b_0$ & $3.2^{c-\alpha+2}b_0$& $2^{n-\alpha}(\delta_cy + 2^{c+2})x_1+2^{-\alpha}b_0(\delta_cy + 6\delta_c +2^cy+ 3.2^{c+2})x_2$ \\
\hline
$0$ &$2^{n-\alpha}$ & $2^{n+c-\alpha-1}$ & $2^{-\alpha}b_0$ &$ 2^{c-\alpha}b_0$ & $ 2^{n+c-\alpha+1}$ & $ 0 $ & $0$& $2^{n-\alpha}(\delta_c + 2^{c-1}y+2^{c+1})x_1+2^{-\alpha}b_0(\delta_cy +2^cy)x_2$ \\
\hline
$0$ &$2^{n-1-\alpha}$ & $2^{n+c-\alpha-1}$ & $2^{-\alpha}b_0$ &$ 2^{c-\alpha}b_0$ & $ 2^{n+c-\alpha}$ & $ 2^{-\alpha}b_0 $ & $2^{-\alpha}b_0$& $2^{n-\alpha}(2^{-1}\delta_c + 2^{c}y+2^{c+1})x_1+2^{-\alpha}b_0(\delta_cy +\delta_c+2^cy+2^{c+1})x_2$ \\
\hline
\hline
$2^{n-\alpha}$ &$2^{n-\alpha}$ & $2^{n+c-\alpha-1}$ & $2^{-\alpha}b_0$ &$ 2^{c-\alpha}b_0$ & $ -2^{n+c-\alpha+1}$ & $ 2^{-\alpha+2} $ & $2^{c-\alpha + 3}$& $2^{n-\alpha}(\delta_cy +\delta_c +2^{c-1}(y - 4))x_1+2^{-\alpha}b_0(\delta_cy +4\delta_c+2^c(y+8))x_2$ \\
\hline
\end{tabular}
\caption{A possible solution set(continued)}
\end{table}
\end{landscape}

In the table below, we've grouped the solution set according to leading terms for easier reference.\\
\noindent
\begin{tabular}{ |p{8.5cm}|p{8.5cm}|  }
\hline
\multicolumn{2}{|c|}{\colorbox[gray]{0.8} {Polynomials with leading term $\bf{\delta_cyx_1}$}}\\
\hline
\multicolumn{2}{|l|}{$ 2^{n-\alpha}(\delta_cy +2^{c-1}y -2^{c+1})x_1 + 2^{-\alpha}b_0(2^2\delta_c + 2^cy + 2^{c+3})x_2$}\\
\multicolumn{2}{|l|}{$2^{n-\alpha}(\delta_cy -2^{c+2})x_1 + 2^{-\alpha+2}b_0(\delta_c +2^{c+1})x_2$} \\
\multicolumn{2}{|l|}{$2^{n-\alpha}(\delta_cy - 3(2^c))x_1 + 2^{-\alpha}b_0(2^2\delta_c+2^cy+2^{c+3})x_2$}\\
\multicolumn{2}{|l|}{$2^{n-\alpha}(\delta_cy -3(2^{c+1}))x_1 +2^{-\alpha}b_0(\delta_cy + 6\delta_c +3.2^{c+2})x_2$} \\
\multicolumn{2}{|l|}{$2^{n-\alpha}(\delta_cy + 2^{c-1}y -3(2^{c}))x_1 + 2^{-\alpha}b_0(2^2\delta_c + 2^{c+3})x_2$}\\
\multicolumn{2}{|l|}{$(\delta_cy+2\delta_c)x_1$}\\
\multicolumn{2}{|l|}{$ (\delta_cy +2\delta_c +2^{c+1}y+2^{c+2})x_1$}\\
\multicolumn{2}{|l|}{$2^{n-\alpha}(\delta_cy -5.2^{c})x_1+2^{-\alpha}b_0(\delta_cy +6\delta_c +2^cy + 3.2^ {c+2} )x_2$}\\
\multicolumn{2}{|l|}{$ 2^{n-\alpha}(\delta_cy +\delta_c -2^c)x_1 + 2^{-\alpha}b_0(2\delta_c + 2^cy + 2^{c+2})x_2$}\\
\multicolumn{2}{|l|}{$ 2^{n-\alpha}(\delta_cy + \delta_c - 2^{c+2})x_1+2^{-\alpha}b_0(\delta_cy +2^2\delta_c + 2^{c+3})x_2$}\\
\multicolumn{2}{|l|}{$2^{n-\alpha}(\delta_cy+2^{c-1}y-5.2^{c})x_1+2^{-\alpha}b_0(\delta_cy - 6\delta_c -3.2^{c+2})x_2$}\\
\multicolumn{2}{|l|}{$2^{n-\alpha}(\delta_cy + \delta_c +2^{c-1}y - 3.2^c)x_1+2^{-\alpha}b_0(\delta_cy + 4\delta_c +2^{c+3})x_2$}\\
\multicolumn{2}{|l|}{$2^{n-\alpha}(\delta_cy + \delta_c +2^{c-1}y)x_1+2^{-\alpha}b_0(2\delta_c +2^cy+ 2^{c+2})x_2$}\\
\multicolumn{2}{|l|}{$2^{n-\alpha}(\delta_cy + \delta_c -3.2^{c})x_1+2^{-\alpha}b_0(\delta_cy + 2^2\delta_c +2^cy+ 2^{c+3})x_2$}\\
\multicolumn{2}{|l|}{$2^{n-\alpha}(\delta_cy + 2^{c-1}y+ 2^{c+2})x_1+2^{-\alpha}b_0(\delta_cy + 6\delta_c +2^cy+ 3.2^{c+2})x_2$}\\
\multicolumn{2}{|l|}{$2^{n-\alpha}(\delta_cy +\delta_c +2^{c-1}y - 2^{c+1})x_1+2^{-\alpha}b_0(\delta_cy +4\delta_c+2^cy+2^{c+3})x_2$}\\
\hline
\hline
\multicolumn{2}{|c|}{ \colorbox[gray]{0.8}{Polynomials with leading term $\bf{ yx_1}$ }}\\
\hline
\multicolumn{2}{|l|}{$(y+2)x_1$} \\
\multicolumn{2}{|l|}{$(y+2^{n-\alpha}+2)x_1+ 2^{-\alpha}b_0yx_2$}\\
\multicolumn{2}{|l|}{$(2^{n+c-\alpha}(2^{-1}y-1))x_1 + 2^{-\alpha}b_0(\delta_cy +2\delta_c + 2^{c+2})x_2 $}\\
\multicolumn{2}{|l|}{$2^{n+c-\alpha-1}yx_1+2^{-\alpha}b_0(\delta_cy +2\delta_c +2^cy+2^{c+2})x_2 $}\\
\hline
\hline
 \multicolumn{2}{|c|}{\colorbox[gray]{0.8}{Polynomials with leading term $ \bf{x_1}$} }\\
\hline
\multicolumn{2}{|l|}{$-2^{n+c-\alpha+1}x_1 + 2^{-\alpha}b_0(\delta_cy -2^{n+c+1}\delta_c + 2^{c+2})x_2$} \\
\multicolumn{2}{|l|}{$2^{n-\alpha}x_1 + 2^{-\alpha}b_0yx_2$}\\
\multicolumn{2}{|l|}{$ -2^{n+c-\alpha}x_1 + 2^{-\alpha}b_0(\delta_cy +2\delta_c +2^cy +2^{c+2})x_2$}\\
\hline
\hline
 \multicolumn{2}{|c|}{\colorbox[gray]{0.8} {Polynomials with leading term $\bf{\delta_cyx_2}$}} \\
\hline
\multicolumn{2}{|l|}{$(\delta_cy +2\delta_c +2^{c+1}y +2^{c+2})x_2$}\\
\hline
\multicolumn{2}{|c|}{ \colorbox[gray]{0.8}{Polynomials with leading term $\bf{\delta_cx_1}$}}\\
\hline
\(\alpha =n\) & \(\alpha < n\)\\
\hline
$2^{n-\alpha}(\delta_c + 2^{c+1})x_1  - 2^{-\alpha}b_0(2\delta_c +2^{c+2})x_2 $& $2^{n-\alpha}(2^{-1}\delta_c + 2^{c})x_1  - 2^{-\alpha}b_0(\delta_c +2^{c+1})x_2 $\\ 
$ 2^{n-\alpha}(\delta_c +3(2^{c}))x_1 + 2^{-\alpha}b_0(-2\delta_c +2^cy -2^{c+2})x_2$& $ 2^{n-1-\alpha}(\delta_c )x_1 - 2^{-\alpha}b_0(\delta_c +2^{c}y +2^{c+1})x_2$\\
$2^{n-\alpha} \delta_cx_1+ 2^{-\alpha}b_0\delta_cyx_2$&$2^{n-1-\alpha} (\delta_c+ 2^{c+1})x_1+ 2^{-\alpha}b_0(\delta_cy +\delta_c +2^{c+1})x_2$\\
$2^{n-\alpha}(\delta_c+2^{c-1}y +3.2^{c})x_1 -2^{-\alpha}b_0(2\delta_c+2^{c+2})x_2 $& $2^{n-\alpha}(2^{-1}\delta_c-2^{c-1}y)x_1 -2^{-\alpha}b_0(\delta_c+2^{c+1})x_2 $\\
$ 2^{n-\alpha}(\delta_c + 2^{c-1}y + 2^c)x_1 + 2^{-\alpha}b_0\delta_cyx_2$& $ 2^{n-\alpha}(2^{-1}\delta_c + 2^{c-1}y)x_1 + 2^{-\alpha}b_0(\delta_cy +\delta_c +2^{c+1})x_2$\\
 $ 2^{n-\alpha}(\delta_c +2^c)x_1+2^{-\alpha}b_0(\delta_cy+2^cy)x_2$& $ 2^{n-\alpha}(2^{-1}\delta_cx_1+2^{-\alpha}b_0(\delta_cy+\delta_c+2^cy +2^{c+1})x_2$\\
 $2^{n-\alpha}(\delta_c + 2^{c+1})x_1+2^{-\alpha}b_0(\delta_cy +2^cy)x_2$& $2^{n-\alpha}(2^{-1}\delta_c + 2^{c}y+2^{c+1})x_1+2^{-\alpha}b_0(\delta_cy +\delta_c+2^cy+2^{c+1})x_2$\\
 $ 2^{n-\alpha}(\delta_c +2^{c-1}y + 2^{c+2})x_1 -2^{-\alpha}b_0(2\delta_c - 2^cy +2^{c+2})x_2$& $ 2^{n-\alpha}(2^{-1}\delta_c -2^{c-1}y + 2^{c})x_1 - 2^{-\alpha}b_0(\delta_c - 2^cy +2^{c+1})x_2$\\
 \hline
 $2^{n-\alpha}\delta_c x_1 - 2^{-\alpha}b_0(2\delta_c +2^{c+1}y +2^{c+2})x_2$&\\
 \hline
 \hline
\end{tabular}

\end{document}